\documentclass[11pt,tightenlines,twoside,onecolumn,nofloats,nobibnotes,nofootinbib,%
superscriptaddress,noshowpacs,centertags]{revtex4-2}
\usepackage{ljm}
\usepackage{xcolor}
\usepackage{amsmath}
\usepackage{amssymb}
\usepackage{hyperref}

\newcommand{\artanh}{\operatorname{artanh}}
\newcommand{\Cc}{\mathbb{C}}
\newcommand{\Ee}{\mathcal{E}}
\newcommand{\geqs}{\geqslant}
\newcommand{\lambmax}{\lambda_{\max}}
\newcommand{\leqs}{\leqslant}
\newcommand{\mmax}{m_{\max}}
\newcommand{\Oo}{\mathcal{O}}
\newcommand{\Rr}{\mathbb{R}}
\newcommand{\tol}{\texttt{tol}}

\newtheorem{remark}{Remark}
\newtheorem{proposition}{Proposition}

\begin{document}
    \titlerunning{Explicit adaptive time stepping for the Cahn--Hilliard equation}
    \authorrunning{Botchev}

    \title{Explicit adaptive time stepping for the Cahn--Hilliard equation 
      by exponential Krylov subspace and\\Chebyshev polynomial methods}

    \author{\firstname{Mikhail A.}~\surname{Botchev}}
    \email[E-mail: ]{botchev@kiam.ru}
    \affiliation{Keldysh Institute of Applied Mathematics of Russian Academy of Sciences, Moscow, 125047 Russia}

    \firstcollaboration{~~~}

\received{January ---, 2025;  
revised ---, 2025; 
accepted ---, 2025}

\begin{abstract}
The Cahn--Hilliard equation has been widely employed within various
mathematical models in physics, chemistry and engineering.
Explicit stabilized time stepping methods can be attractive for 
time integration of the Cahn--Hilliard equation, especially
on parallel and hybrid supercomputers.  
In this paper, we propose an exponential time integration method for the 
Cahn--Hilliard equation and describe its efficient Krylov subspace based
implementation.  We compare the method to a Chebyshev polynomial
local iteration modified (LIM) time stepping scheme.  Both methods
are explicit (i.e., do not involve linear system solution)
and tested with both constant and adaptively chosen time steps.
\end{abstract}

\maketitle

\section{Introduction}
The Cahn-Hilliard equation has been instrumental in various 
physical, chemical and engineering applications, such as 
multiphase fluid dynamics~\cite{Feng_ea2020,XiaKimLi2022},
image inpainting~\cite{CherfilsFakihMiranville2016},
quantum dot formation~\cite{WiseLowengrubKimJohnson2004},
and flow visualization~\cite{Garcke_ea2001}.
The need to solve the Cahn--Hilliard efficiently has triggered 
developments in various fields of scientific computing,
in particular, in
time integration 
methods~\cite{Eyre1998,Eyre1998a,VollmayrleeRutenberg2003,WiseWangLowengrub2009,TierraGGonzalez2015},
different types of finite element methods~\cite{WellsKuhlGarikipati2006,KayWelford2006,GuoXu2014,ShenYang2010},
meshless methods~\cite{DehghanMohammadi2015}, 
and other scientific computing fields.

An important property of the Cahn--Hilliard equation is
typically a large  discrepancy of the involved time scales,
which leads to numerical stiffness. 
Nonlinearity and stiffness of the Cahn--Hilliard equation 
make its numerical time integration a highly challenging task.
Characteristic time scales in the equation can
vary from $\epsilon^2$ at initial evolution times
(where $0<\epsilon\ll 1$ is defined below in~\eqref{CH},\eqref{eps})
to $1/\epsilon$ at later evolution stages.
In addition, for some values of $\epsilon$ 
(see  discussion below relation~\eqref{eps}),
the norm of the Jacobian matrix in the equation is proportional
to $h^{-4}$, with $h$ being a typical space grid size.
Stiffness can usually be well handled by implicit schemes.  However,
employing implicit time integration schemes for
the Cahn--Hilliard equation is not straightforward.
Indeed, solutions of the Cahn--Hilliard equation may exhibit
a growth in a certain norm, so that fully implicit time stepping
schemes may not be stable~\cite{StuartHumphries1994,Eyre1998,Eyre1998a}
and/or uniquely resolvable in the sense that the next time step
solution may not be unique~\cite{VollmayrleeRutenberg2003,WiseWangLowengrub2009,TierraGGonzalez2015}.

Thus, to be efficient for the Cahn--Hilliard equation, a time integration scheme 
should allow sufficiently large time steps and be adaptive, so that 
eventual accuracy or stability loss can properly be handled.
A majority of time integration schemes being employed nowadays
for the Cahn--Hilliard equation
are (partially) implicit, see, 
for instance,~\cite{Lee_ea2013,TierraGGonzalez2015}.
However, solving one or several (non)linear systems per time step
can become a severe obstacle for implementing implicit schemes on
modern multiprocessor or heterogeneous computers.
This is because parallelism in direct and preconditioned iterative linear
solvers is inherently restricted on distributed memory systems.
A compromise between stability and implementation simplicity
can be provided by explicit stabilized time integration 
methods (see, e.g.,~\cite{LokLokDAN,ShvedovZhukov,RKC97,Lebedev98,MRAIpap,RKC2004})
and we explore this possibility here.
We note that apparently one of the fist known explicit stabilized
schemes, based on Chebyshev polynomials, is due to Yua\'n \v{C}\v{z}ao-Din 
who did his PhD research with L.A.~Lyusternik and V.K.~Saul'ev 
at Moscow State University in the 1950s~\cite{CzaoDin1957, CzaoDin1960}.

In this paper, we propose an exponential time integration 
scheme for the Cahn--Hilliard equation.  The scheme is explicit
in the sense that it requires computing only matrix-vector products
and involves no linear system solutions.  
It is based on efficient restarted Krylov subspace evaluation
of the relevant matrix function (more precisely, 
the $\varphi$ function defined below in~\eqref{phi})~\cite{ART,BKT21}.
We show that this scheme has a nonstiff order~2 and design 
an adaptive time stepping strategy for the scheme.
We refer to the proposed scheme as EE2, exponential Euler scheme of order~2.
Another small contribution of this paper is that
the restarting procedure for the Krylov subspace method, the so called
residual-time (RT) restarting~\cite{ART,BKT21},
is naturally extended to a nonlinear setting.
This allows to significantly improve efficiency of the EE2 scheme.
Furthermore, the proposed EE2 scheme 
is compared numerically with the local iteration modified (LIM) 
time stepping scheme~\cite{LokLokDAN,Zhukov2011}.

The LIM scheme is an explicit time integration scheme where stability
is achieved by carrying out a certain number of Chebyshev iterations.
Algorithmically, the Chebyshev iterations in the LIM scheme can be viewed 
as steps of the simplest explicit Euler scheme with specially chosen
parameters.  Therefore, the LIM scheme is very easy to implement, also
in parallel and heterogeneous environments.
However, unlike the explicit Euler scheme,
the time step size in the LIM scheme is not restricted by the severe
stability restrictions.
For instance, in parabolic problems the time step size $\tau$
of the explicit Euler scheme is subject to the CFL stability restriction
$\tau\leqs\tau_{\max}=\Oo(h^2)$.  
Hence, to integrate the problem on a time interval $t\in[0,T]$,
the explicit Euler scheme needs at least $T/\tau_{\max} = \Oo(Th^{-2})$ time steps.
In contrast, as will become clear from further discussion
(see relations~\eqref{p},\eqref{p_est}), for the same
problem the LIM scheme requires $\sqrt{\Oo(T h^{-2})} = \Oo(\sqrt{T} h^{-1})$
iterative steps, each equal in cost to an explicit Euler step.
In addition, the LIM scheme is monotone~\cite{BotchevZhukov2025}, i.e.,
it preserves solution nonnegativity.
The LIM scheme has recently been shown to work for the Cahn--Hilliard
equation successfully~\cite{BotchevFahurdinovSavenkov2024}, also in combination
with a variant of the Eyre convex splitting.  
In~\cite{BotchevFahurdinovSavenkov2024} the LIM scheme is applied 
with a constant time step size.
To get an adaptive time step size selection in the LIM scheme, here we adopt the
strategy proposed in~\cite{BotchevZhukov2024}. 

For both EE2 and LIM time integration schemes it is discussed how
these schemes can be combined with a variant of the Eyre convex splitting,
and it is shown that the splitting leads to a matrix similar 
to a symmetric positive semidefinite matrix.

The paper is organized as follows. Section~\ref{s:prob} is devoted
to the problem setting.  The LIM scheme is described in Section~\ref{s:LIM},
and, in particular, its combination with the Eyre convex splitting
in Section~\ref{s:LIMcs} and an adaptive time step selection in 
Section~\ref{s:LIMadapt}.  Section~\ref{s:EE} presents the EE2 scheme,
including its Krylov subspace implementation (Section~\ref{s:EERT}),
analysis of its consistency order (Section~\ref{s:accEE}),
its combination with the Eyre splitting (Section~\ref{s:EEcs})
and its adaptive time stepping variant (Section~\ref{s:EEadapt}).
Numerical experiments are discussed in Section~\ref{s:num} and
conclusions are drawn in Section~\ref{s:concl}.

Unless indicated otherwise, throughout the paper $\|\cdot\|$ denotes 
the Euclidean vector or matrix norm.

\section{Problem setting}
\label{s:prob}
We consider the Cahn-Hilliard equation, posed in for an unknown function
$c(x,t)$, with $(x,t)\in\Omega\times[0,T]$, $\Omega\subset\Rr^d$, $d=1,2,3$,
and accompanied by initial and homogeneous Neumann boundary conditions: 
\begin{equation}
\label{CH}
\begin{aligned}
\frac{\partial c(x,t)}{\partial t} 
&= \Delta\mu(c(x,t)),
\\
\mu(c(x,t)) &= F'(c(x,t)) - \epsilon^2 \Delta c(x,t),
\\
\frac{\partial c}{\partial n} &= \frac{\partial\mu}{\partial n} = 0
\quad \text{for} \quad x\in\partial\Omega,
\\
c(x,0) &= c_0(x).
\end{aligned}
\end{equation}
Here the sought after function $c(x,t)$ denotes a difference of two mixture 
concentrations, $\mu(c(x, t)$ is the chemical potential,
$F(c)$ is the Helmholtz free energy defined as
\begin{equation}
\label{F}
F(c) = \frac14( 1 - c^2 )^2 , \quad
F'(c) = c( c^2 - 1 ),
\end{equation}
$\epsilon>0$ is a small parameter related to the interfacial energy,
$n$ is the outward unit normal vector to $\partial\Omega$, and
$c_0(x)$ is the given initial mixture concentration difference.
Two important characteristic properties of the Cahn--Hilliard equation
are that a total energy functional $\Ee(c)$ does not increase with time
and the total mass is preserved:
\begin{align}
\label{Ee}
& \frac{d}{d t}\Ee(c(t)) \leqs 0 \quad\text{for}\quad
\Ee(c) := \int_\Omega \left(F(c) + \frac{\epsilon^2}2|\nabla c|^2\right)d\,x,
\\
\label{m}
& \frac{d}{d t} \int_\Omega c(x,t) d\,x = 0.
\end{align}

We assume that a space discretization is applied to the 
initial boundary value problem~\eqref{CH}, leading to a initial value
problem (IVP) 
\begin{equation}
\label{IVP}
\begin{aligned}
y'(t) &= -A \left( F'(y(t)) + \epsilon^2 A y(t) \right),
\\
y(0) &= y^0,
\end{aligned}
\end{equation}
where the components $y_j(t)$ of the unknown vector function
$y: \Rr\rightarrow\Rr^N$ contain spatial grid values of the
mixture concentration difference $c(x,t)$ at time moment~$t$
and $A\in\Rr^{N\times N}$ is a symmetric positive semidefinite matrix
such that $-A$ is a Laplacian discretization. 
The value of $\epsilon$ is often chosen as~\cite{Lee_ea2013} 
\begin{equation}
\label{eps}
\epsilon := \epsilon_m = \frac{h m}{2\sqrt{2} \artanh(9/10)}, 
\end{equation}
where $h$ is the space grid step and $m$ is the number of grid points 
on which the concentration $c(x,t)$ changes from its minimum
its maximum (or vice versa).
Since for typical Laplacian discretizations we have $\|A\|=\Oo(h^{-2})$,
we should expect $\epsilon^2 \|A\|=\Oo(1)$ for $\epsilon$ chosen
as in~\eqref{eps}.  Then the norm of the right hand side
function $-A \left((F'(y(t)) + \epsilon^2 A y(t)\right)$ in~\eqref{IVP}
is $\Oo(h^{-2})$ and the problem~\eqref{IVP} is moderately stiff.
However, if it is desirable to refine the space 
grid while keeping the value of~$\epsilon$ fixed, we will have
$\epsilon^2 \|A\|=\Oo(h^{-2})$.  In this case the norm of the right hand side
function $-A \left((F'(y(t)) + \epsilon^2 A y(t)\right)$ in~\eqref{IVP}
will be proportional to~$h^{-4}$ and the problem~\eqref{IVP} can be 
considered as highly stiff. 

\begin{remark}
\label{rmkM}
For the Cahn--Hilliard equation
\begin{equation}
\label{CHmob}  
\frac{\partial c(x,t)}{\partial t} = \nabla\left(M(c)\nabla\mu(c(x,t))\right),
\end{equation}
where $M(c)$ is the diffusional mobility, the space discretized equation 
in~\eqref{IVP} takes form
$$
y'(t) = -\tilde{A}(y) \left( F'(y(t)) + \epsilon^2 A y(t) \right),
$$  
with $A_M(y)y$ being a discretization of the diffusion
operator $\nabla(M(c)\nabla c)$.
The methods considered in this paper are directly applicable to
the equation~\eqref{CHmob} provided that
$A_M(y)$ can be taken constant during a time step, i.e.,
taken, for instance, to be $A_M(y^n)$. 
Then it suffices to define the matrix $\widehat{A}_n$
below in~\eqref{A_g} as
$$
\widehat{A}_n = A_M(y^n) \left( J_n + \epsilon^2 A \right).
$$
If the Eyre convex splitting is included,
a similar modification should then also be done in~\eqref{A_g_cs}.
\end{remark}

\section{LIM (local iteration modified) time integration scheme}
\label{s:LIM}
\subsection{Description of the LIM scheme}
For the implicit Euler method for~\eqref{IVP}
\begin{equation}
\label{EB}
\frac{y^{n+1} - y^n}{\tau} = -A \left(F'(y^{n+1}) + \epsilon^2 Ay^{n+1}\right),  
\end{equation}
with $\tau>0$ being the time step size,
apply a linearization
\begin{equation}
\label{lin}
F'(y^{n+1}) \approx F'(y^n) + J_n(y^{n+1} - y^n)  
\end{equation}
and consider its linearized version
\begin{gather}
\label{EBlin}
\frac{y^{n+1} - y^n}{\tau} = 
-\widehat{A}_n y^{n+1} + \widehat{g}^n, \quad \text{with}
\\
\label{A_g}
\widehat{A}_n = A \left( J_n + \epsilon^2 A \right), \quad
\widehat{g}^n = A\left( J_n y^n - F'(y^n) \right).
\end{gather}
The matrix $J_n\in\Rr^{N\times N}$ in~\eqref{lin} is the Jacobian of the mapping 
$F'(y)$ (cf.~\eqref{F}) evaluated at $y=y^n$:
\begin{equation}
\label{J}
J_n = \frac{\partial}{\partial y} F'(y^n) = 
3 \operatorname{diag} \left(\, (y_1^n)^2, \, \dots,\, (y_N^n)^2 \,\right) - I,
\end{equation}
where $\operatorname{diag} (a_{11},\dots,a_{NN})$ denotes a diagonal matrix
with entries $a_{11}$, \dots, $a_{NN}$ 
and $I$ is the $N\times N$ identity matrix.

In the LIM scheme 
the linear system solution needed
to compute $y^{n+1}$ in the linearized implicit Euler scheme~\eqref{EBlin}
is replaced by $2p-1$ Chebyshev iterations, each of which can be seen as
an explicit Euler step with a specially chosen time step size.
These $2p-1$ iterations, combined, form actions of a Chebyshev polynomial 
to guarantee a time integration stability.
The value of~$p$, which determines the Chebyshev polynomial order, 
is chosen based on the time step size~$\tau$
and an upper spectral bound $\lambmax\approx\|\widehat{A}_n\|$ as
\begin{equation}
\label{p}
p = \left\lceil\dfrac{\pi/4}{\pi/2-\arctan\sqrt{\tau\lambmax}}\right\rceil,
\end{equation}
where $\lceil x \rceil$ is the ceiling function defined as the smallest
integer greater than or equal to~$x$.
We have 
\begin{equation}
\label{p_est}
\begin{gathered}
\dfrac{\pi/4}{\pi/2-\arctan\sqrt{\tau\lambmax}} \leqs 
\frac{\pi}4\sqrt{\tau\lambmax + 1},
\\
\frac{\pi}4\sqrt{\tau\lambmax + 1}
-\frac{\pi/4}{\pi/2-\arctan\sqrt{\tau\lambmax}} \rightarrow 0,
\qquad \tau\lambmax\rightarrow\infty,
\end{gathered}
\end{equation}
so that $p\sim \sqrt{\tau \,\lambmax}$.  

An algorithmic description of the LIM scheme time step
$y^n\rightarrow y^{n+1}$, replacing the linearized implicit Euler
step~\eqref{EBlin}, is presented in Figure~\ref{f:LIM}.
As we see, the main steps in the algorithm, i.e., steps~6 and~10, can be seen
as explicit Euler steps.  Note that the scheme is quite simple and 
does not require any parameter tuning.  Furthermore, note that $a_1=0$
and therefore,  
in case $\tau$ is taken so small that relation~\eqref{p} sets $p=1$, 
the LIM scheme is reduced to the explicit Euler scheme.

\begin{figure}
\begin{center}
\begin{tabular}{rl}
\hline\hline
\multicolumn{2}{l}{A LIM scheme time step $y^n\rightarrow y^{n+1}$ 
for IVP $y'(t)=-\widehat{A}_ny(t) + \widehat{g}^n$, $y(0)=y^n\phantom{\biggl|}$.}
\\
\multicolumn{2}{l}{\textbf{Given:} 
time step size $\tau>0$, 
$\widehat{A}_n\in\Rr^{N\times N}$, $\widehat{g}^n\in\Rr^N$, $y^n\in\Rr^N$.}
\\
\multicolumn{2}{l}{\textbf{Output:} 
solution $y^{n+1}$ at time step~$n+1$.}
\\\hline
1. & Determine a spectral upper bound 
     $\lambmax:=\|\widehat{A}_n\|_1=\max_j\sum_i|\widehat{a}_{ij}|$
\\
   & and Chebyshev polynomial order~$p$ by~\eqref{p}. 
\\
2. & Determine an ordered set 
$\{ \beta_m, \; m=1,\ldots,p \} = \left\{ \cos\pi\frac{2i-1}{2p}, \; i=1,\dots,p\right\}$, 
\\
   & with $\beta_1=\cos({\pi}/{2p})$.
\\
3. & $z_1:=\beta_1$, $y^{\text{(prev)}}:=y^n$.
\\
4. & For $m=1,\dots,p$
\\
5. & \hspace*{2em}
$a_m:=\dfrac{\lambmax}{1+z_1}(z_1 - \beta_m)$
\\
6. & \hspace*{2em}
$y^{\text{(next)}} := \dfrac1{1+\tau a_m}\left(
y^n + \tau a_my^{\text{(prev)}} + \tau(\widehat{g}^n - \widehat{A}_n y^{\text{(prev)}})
\right),$
\\
7. & \hspace*{2em} $y^{\text{(prev)}}:=y^{\text{(next)}}$.
\\
8.  & end for
\\
9. & For $m=2,\dots,p$
\\
10. & \hspace*{2em} 
$y^{\text{(next)}} := \dfrac1{1+\tau a_m}\left(
y^n + \tau a_my^{\text{(prev)}} + \tau(\widehat{g}^n - \widehat{A}_n y^{\text{(prev)}})
\right),$
\\
11. & \hspace*{2em} $y^{\text{(prev)}}:=y^{\text{(next)}}$. 
\\
12. & end for
\\
13. & Return solution $y^{n+1}:=y^{\text{(next)}}$.
\\\hline
\end{tabular}
\end{center}
\caption{A time step of the LIM scheme based on the
linearized implicit Euler scheme~\eqref{EBlin}, 
with~$\widehat{A}_n$ and~$\widehat{g}^n$ defined by
either~\eqref{A_g} or~\eqref{A_g_cs}}
\label{f:LIM}
\end{figure}

For parabolic problems
$\lambmax\sim h^{-2}$ and, hence, to integrate a parabolic problem
for $t\in[0,T]$, from the stability point of view, it is suffice to 
carry out $\sqrt{\Oo(T h^{-2})} = \Oo(\sqrt{T} h^{-1})$ Chebyshev steps.  
For the Cahn--Hilliard equation, when the value~$\epsilon$ in~\eqref{eps}
is fixed with respect to the space step~$h$,
we have $\|\widehat{A}_n\|=\Oo(h^{-4})$ and  
\begin{equation*}
p = \Oo(\sqrt{T} h^{-2}).  
\end{equation*}
In contrast, applied to the Cahn--Hilliard equation, the explicit Euler scheme
would have suffered from the CFL stability restriction 
$\tau\leqs\tau_{\max}=\Oo(\|\widehat{A}_n\|^{-1})=\Oo(h^4)$
and needed $T/\tau_{\max}=\Oo(Th^{-4})$ steps to integrate the problem
for $t\in[0,T]$.    

The construction of Chebyshev polynomials in the LIM scheme
requires that the matrix $\widehat{A}_n$ is symmetric positive
semidefinite.  Although the matrix $\widehat{A}_n$ in~\eqref{A_g}
is not symmetric, the LIM scheme technique turns out to be applicable
to our problem with~$\widehat{A}_n$ because, as shown
by Proposition~\ref{p:symm} to be proven next, $\widehat{A}_n$ can be shown
to be similar to a symmetric matrix $\tilde{A}_n$.  
(Recall that matrices $A, B\in\Rr^{N\times N}$ are called similar
if there exists a nonsingular matrix $C\in\Rr^{N\times N}$ such that
$A=CBC^{-1}$.)
Then observe that the Chebyshev matrix polynomial in $\widehat{A}_n$,
implicitly constructed by the LIM scheme, turns out to be similar
to the same Chebyshev polynomial in the symmetric matrix $\tilde{A}_n$.
This observation resolves the Chebyshev polynomial 
applicability issue only partially, as $\tilde{A}_n$ may not
be positive semidefinite.  This lack of positivity appears to be closely 
connected to the well known stability problems for the Cahn--Hilliard 
equation, see, e.g.,~\cite{VollmayrleeRutenberg2003,WiseWangLowengrub2009},
and can be solved by combining the LIM scheme 
with the Eyre convex splitting, see Section~\ref{s:LIMcs} below.
We remark that in numerical tests presented here we have not experienced 
any serious stability problems and, hence, there was no need to apply 
the convex splitting.

\begin{proposition}
\label{p:symm}
Let $A\in\Rr^{N\times N}$ defined in~\eqref{IVP} 
be a symmetric positive semidefinite matrix.
Then the matrix~$\widehat{A}_n$ defined in~\eqref{A_g} is diagonalizable,
has real eigenvalues and the number of its negative (positive) eigenvalues
does not exceed the number of negative (positive) eigenvalues of the 
matrix $J_n + \epsilon^2 A$.  Furthermore,
there exists a symmetric matrix $\tilde{A}_n$ which is similar 
the matrix~$\widehat{A}_n$.
Therefore the LIM Chebyshev matrix polynomial in~$\widehat{A}_n$ is a matrix 
similar to the same polynomial in~$\tilde{A}_n$.  
\end{proposition}

\begin{proof}
Observe that $\widehat{A}_n$ is a product of the symmetric
positive semidefinite matrix $A$ and the symmetric matrix $J_n + \epsilon^2 A$.
Then, by Theorem~7.6.3 and Exercise~7.6.3 from~\cite{HornJohnsonI},
we have that $\widehat{A}_n$ is a diagonalizable matrix with 
real eigenvalues and can not have more positive (respectively,
negative) eigenvalues than $J_n + \epsilon^2 A$ has.  The second
part of Theorem~7.6.3 in~\cite{HornJohnsonI} states that there exist
a symmetric positive definite matrix $\bar{A}$ and a
symmetric matrix $B$ such that $\widehat{A}_n=\bar{A}B$.
Note that $\widehat{A}_n$ is similar to a symmetric matrix
$$
\bar{A}^{-1/2}\widehat{A}_n\bar{A}^{1/2} = 
\bar{A}^{-1/2}\bar{A}B\bar{A}^{1/2} = \bar{A}^{1/2}B\bar{A}^{1/2}, 
$$
which is the sought after matrix~$\tilde{A}_n$.  
\end{proof}

As shown in~\cite[Theorem~3]{Zhukov2011}, for linear IVPs
$y'(t)=-\widehat{A}_ny(t) + \widehat{g}^n$, $y(0)=y^n$, 
the LIM scheme has accuracy order~1.  Since linearization~\eqref{lin}
of~\eqref{IVP} introduces an error $\Oo(\|y^{n+1} - y^n\|^2)$
to the system,
it is straightforward to conclude that the presented LIM scheme
has accuracy order~1.  For further details on the LIM scheme
we refer to~\cite{Zhukov2011,BotchevFahurdinovSavenkov2024,BotchevZhukov2025}.

\subsection{The LIM scheme combined with the Eyre convex splitting}
\label{s:LIMcs}
To avoid stability problems which may occur when implicit time integration 
schemes are applied to the Cahn--Hilliard equation with large time step
sizes, the schemes are usually applied in combination with the well known
convex splitting proposed by Eyre~\cite{Eyre1998,Eyre1998a}.  
One popular combination of the implicit Euler scheme with convex
splitting is the non-linearly stabilized splitting (NLSS) scheme~\cite{Lee_ea2013}
which reads
\begin{equation}
\label{EBcs}
\frac{y^{n+1} - y^n}{\tau} = -A 
\left(F'(y^{n+1}) + \epsilon^2 Ay^{n+1} + y^{n+1} - y^n \right).  
\end{equation}
Linearizing here $F'(y^{n+1})$ as in~\eqref{lin}, we arrive at a scheme
which can be seen as an implicit Euler step~\eqref{EBlin} with
\begin{equation}
\label{A_g_cs}
\widehat{A}_n = A \left( J_n + \epsilon^2 A + I\right), \quad
\widehat{g}^n = A\left( J_n y^n + y^n - F'(y^n) \right).  
\end{equation}
Comparing this $\widehat{A}_n$ to the one defined
in~\eqref{A_g}, we see that the convex splitting introduces an additional
identity matrix term $I$ to the matrix $\widehat{A}_n$.
Note that this linearized implicit Euler scheme~\eqref{EBlin},\eqref{A_g_cs}
scheme is different from the linearly stabilized splitting (LSS) 
scheme~\cite{Lee_ea2013}, which has also been shown to
be successful in combination with LIM for the Cahn--Hilliard
equation~\cite{BotchevFahurdinovSavenkov2024}.

A combination of the LIM scheme with the convex splitting is 
obtained by replacing relation~\eqref{A_g} with~\eqref{A_g_cs}.
The following simple result holds, which guarantees applicability
of the LIM scheme and solvability of the linearized implicit 
Euler scheme~\eqref{EBlin},\eqref{A_g_cs}.

\begin{proposition}
\label{p:stab} 
Let $A\in\Rr^{N\times N}$ defined in~\eqref{IVP} 
be a symmetric positive semidefinite 
matrix.
Then the 
matrix~$\widehat{A}_n$ defined in~\eqref{A_g_cs} is diagonalizable
and has real nonnegative eigenvalues.  Furthermore,
there exists a symmetric positive semidefinite matrix $\tilde{A}_n$ which 
is similar to the matrix~$\widehat{A}_n$.
Therefore the LIM Chebyshev matrix polynomial in~$\widehat{A}_n$ is a matrix 
similar to the same polynomial in~$\tilde{A}_n$.  
In addition,
the matrix $I + \tau\widehat{A}_n$, $\tau>0$ is nonsingular,
so that the linearized implicit Euler scheme~\eqref{EBlin},\eqref{A_g_cs}
is solvable for any $\tau>0$.
\end{proposition}

\begin{proof}
Note that the matrix 
$$
J_n + \epsilon^2 A + I = 
3 \operatorname{diag} \left(\, (y_1^n)^2, \, \dots,\, (y_N^n)^2\right)  
+ \epsilon^2 A
$$  
is symmetric positive semidefinite as it is a sum of 
a symmetric positive semidefinite matrix and a diagonal matrix with
nonnegative entries.  
Then observe that $\widehat{A}_n$ is a product of the symmetric
positive semidefinite matrices $A$ and $J_n + \epsilon^2 A + I$.
Then Theorem~7.6.3 and Exercise~7.6.3 from~\cite{HornJohnsonI}
state that $\widehat{A}_n$ is a diagonalizable matrix with 
real nonnegative eigenvalues.  
The rest of the proof repeats the steps in the proof of
Proposition~\ref{p:symm}.
The last proposition statement, about nonsingularity of the matrix
$I + \tau\widehat{A}_n$, evidently follows from the observation that
all the eigenvalues of $I + \tau\widehat{A}_n$ are real and greater than
or equal to~1.
\end{proof}

\begin{remark}
It is easy to see that Propositions~\ref{p:symm} and~\ref{p:stab} remain valid 
in the case of a nonconstant mobility~$M(c)$, provided that the matrix $A_M(y^n)$, 
defined in Remark~\ref{rmkM}, is symmetric positive semidefinite.   
\end{remark}

\subsection{Adaptive time step size selection in the LIM scheme}
\label{s:LIMadapt}
Here we adopt the time step selection procedure proposed and successfully 
tested for the LIM scheme in~\cite{BotchevZhukov2024}.
It is based on an aposteriori error estimation obtained with
the predictor-corrector approach.
More specifically, let $y^{n+1}$ be the LIM scheme solution and
consider the following predictor-corrector (PC) time stepping, with
the LIM and the implicit trapezoidal (ITR) schemes taken as predictor
and corrector, respectively: 
\begin{equation}
\label{LIMest}
\begin{aligned}
y_{\text{PC}}^{n+1} &= y^n - \frac{\tau}2 A 
\left( F'(y^n) + F'(y^{n+1}) + \epsilon^2 A (y^n+y^{n+1}) \right),
\\
\texttt{est}_{n+1} &= \frac{\|y^{n+1} - y_{\text{PC}}^{n+1}\|}{\|y_{\text{PC}}^{n+1}\|}.
\end{aligned}
\end{equation}
Since LIM has accuracy order~1, one less than the ITR accuracy order~2, 
their PC combination preserves the accuracy order of the corrector, i.e., 
has order~2 (see, e.g.,~\cite{ODE1(HNW)}).
Therefore, the local error estimate should be $\texttt{est}_{n+1}=\Oo(\tau^2)$ 
for sufficiently small $\tau$.  In practice, for realistic $\tau$ values
we observe $\texttt{est}_{n+1}=\Oo(\tau)$ and, hence, choose
a time step size for the next time step as 
\begin{equation}
\label{LIMdt} 
\tau_{\text{new}} :=
\frac{\tol}{\texttt{est}_{n+1}}\,\tau,
\end{equation}
where $\tol$ is a given tolerance value.
The time step adjustment is carried out aposteriori, once a time step is done.
The resulting algorithm for the LIM time integration scheme with
the adaptive time step selection is given in Figure~\ref{f:LIMadapt}.
Another adaptive time stepping approach for the LIM scheme
is recently proposed in~\cite{FeodoritovaNovikovaZhukov2025}.

\begin{figure}
\begin{center}
\begin{tabular}{rl}
\hline\hline
\multicolumn{2}{l}{The LIM adaptive time stepping scheme
for Cahn--Hilliard equation~\eqref{IVP}}
\\[2ex]
\multicolumn{2}{l}{\textbf{Given:} 
initial time step size $\tau_0>0$, final time $T$,
$A\in\Rr^{N\times N}$, $y^0\in\Rr^N$, tolerance~\tol.}
\\
\multicolumn{2}{l}{\textbf{Output:} 
numerical solution $y_{\text{fin}}$ at time~$T$.}
\\\hline
1.  & Set $t:=0$, $\tau:=\tau_0$, $y^n:=y^0$ \\
2.  & While $t<T$ do \\
3.  & \hspace*{2em} if $t+\tau > T$ \\
4.  & \hspace*{4em}    $\tau:= T-t$ \\
5.  & \hspace*{2em} end if          \\
6.  & \hspace*{2em} Compute $\widehat{A}_n$ and $\widehat{g}^n$ 
                    by relation~\eqref{A_g} or, if the Eyre splitting is used,
                    by~\eqref{A_g_cs}\\
7.  & \hspace*{2em} Compute $y^{n+1}$ by making the LIM step (Figure~\ref{f:LIM}), 
                    with $\tau$, $\widehat{A}_n$, $\widehat{g}^n$, $y^n$ \\
8.  & \hspace*{2em} Compute $y_{\text{PC}}^{n+1}$ and $\texttt{est}_{n+1}$ 
                    by relation~\eqref{LIMest}\\
9.  & \hspace*{2em} Set $y^{n}:= y^{n+1}$, $t:=t + \tau$ \\
10. & \hspace*{2em} Set $\tau:=\tau_{\text{new}}$ by relation~\eqref{LIMdt}\\
11. & end while\\ 
12. & Return solution $y_{\text{fin}}:=y^{n+1}$
\\\hline
\end{tabular}
\end{center}
\caption{The LIM adaptive time stepping scheme}
\label{f:LIMadapt}
\end{figure}

\section{EE2 (exponential Euler order 2) time integration scheme}
\label{s:EE}
\subsection{EE2 scheme, Krylov subspace residual and restarting}
\label{s:EERT}
For $y^n\approx y(t_n)$ being a numerical solution 
of the space discretized Cahn--Hilliard IVP~\eqref{IVP} at time $t_n$,
consider a linearization of $F'(y(t))$ in IVP~\eqref{IVP} similar to~\eqref{lin}:
\begin{equation}
\label{linEE}
F'(y(t)) \approx F'(y^n) + J_n(y(t) - y^n),
\end{equation}
with Jacobian $J_n$ defined in~\eqref{J}.
Substituting this linearization into~\eqref{IVP}, we obtain
a linear IVP, which can be solved for $t\in[t_n,t_n+\tau]$, 
to obtain an approximate solution $y^{n+1}$ at $t_{n+1}=t_n+\tau$:
\begin{align}
\label{IVPlin}
& \text{solve}\quad\left\{
\begin{aligned}
y'(t) &= -\widehat{A}_n y(t) + \widehat{g}^n,
\quad t\in[t_n,t_n+\tau],
\\
y(t_n) &= y^n,
\end{aligned}\right.
\\
\label{EE2}
& \text{set}\quad y^{n+1}:=y(t_n+\tau),
\end{align}
where $\widehat{A}_n$ and $\widehat{g}^n$ are given by~\eqref{A_g}.
We call this time integration scheme exponential Euler scheme of 
order~2 (EE2). 
In the EE2 scheme, solution $y(t)$ of~\eqref{IVPlin} is computed by an action
of a matrix function~$\varphi$ defined as
\begin{equation}
\label{phi}
\varphi (z) := \frac{e^z - 1}{z}, \quad z\in\Cc,
\quad \varphi(0):=1.  
\end{equation}
More precisely, solution of IVP~\eqref{linEE} can be written as
a matrix-vector product of the matrix function~$\varphi$
with vector $\widehat{g}^n  - \widehat{A}_n y^n$, 
\begin{equation}
\label{yt}
y(t_n+t) = y^n + t \varphi( -t\widehat{A}_n ) 
\left( \widehat{g}^n  - \widehat{A}_n y^n\right), \quad t\geqs 0,
\end{equation}
and we solve IVP~\eqref{linEE} by evaluating this matrix-vector
product with a Krylov subspace method, see, 
e.g.,~\cite{DruskinKnizh89,GallSaad,HochLub97}.
In particular, we use the residual-time (RT) restarted Krylov
subspace method proposed in~\cite{ART,BKT21}.  An implementation 
of the method is available at \url{https://team.kiam.ru/botchev/expm/}.
A well known, more sophisticated implementation of the Krylov subspace 
method for evaluating the $\varphi$ matrix function is~\cite{EXPOKIT}.

In the Krylov subspace method the Arnoldi process is employed.
For the starting vector $v_1:=\widehat{g}^n  - \widehat{A}_n y^n$,
$v_1:=v_1/\|v_1\|$,
the Arnoldi process computes, after $m$ steps, a matrix
$V_{m+1}\in\Rr^{N\times(m+1)}$,
with orthonormal columns $v_1$, \dots, $v_{m+1}$,  
and an upper Hessenberg matrix $H_{m+1,m}\in\Rr^{(m+1)\times m}$ 
such that (see, e.g.,~\cite{SaadBook,vdVBook})
\begin{equation}
\label{Arn}
\widehat{A}_n V_m = V_{m+1}H_{m+1,m} \quad \text{or}\quad
\widehat{A}_n V_m = V_mH_{m,m} + h_{m+1,m}v_{m+1}e_m^T,
\end{equation}
where $H_{m,m}\in\Rr^{m\times m}$ is formed by the first $m$~rows of $H_{m+1,m}$
and $e_m=(0,\dots,0,1)^T\in\Rr^m$.
Then the Krylov subspace solution $y_m(t_n+t)$,
an approximation to $y(t_n+t)$ in~\eqref{yt}, is computed as
\begin{equation}
\label{yt_m}   
          y(t_n+t) = y^n + t \varphi( -t\widehat{A}_n )V_m(\beta e_1) 
\approx y_m(t_n+t) = y^n + t V_m \varphi( -t H_{m,m} )(\beta e_1) ,
\end{equation}
where $\beta=\|v_1\|=\|\widehat{g}^n  - \widehat{A}_n y^n\|$ and 
$e_1=(1,0,\dots,0)^T\in\Rr^m$.
Since we keep $m\ll N$, 
the matrix function $\varphi( -t H_{m,m} )$
is computed by a direct method, see, e.g.,~\cite{Higham_bookFM}.
The quality of the approximation $y(t_n+t)\approx y_m(t_n+t)$ 
can be easily monitored as follows. Define a residual $r_m(t_n+t)$
of $y_m(t_n+t)$ with respect to the IVP~\eqref{IVPlin} as
\begin{equation}
\label{rm}
r_m(t_n+ t) := -\widehat{A}_n y_m(t_n+t) + \widehat{g}^n - y_m'(t_n+t), 
\quad t\geqs 0.
\end{equation}
Using~\eqref{Arn} it is not difficult to check 
that~\cite{CelledoniMoret97,DruskinGreenbaumKnizhnerman98,BKT21}
$$
\begin{aligned}
r_m(t_n+ t) &= -h_{m+1,m}v_{m+1}e_m^T t \varphi( -t H_{m,m} )(\beta e_1),
\\
\|r_m(t_n+ t)\| &= \beta h_{m+1,m} t | e_m^T \varphi( -t H_{m,m} ) e_1|.
\end{aligned}
$$

Note that in this algorithm $m+1$ columns of $V_m$ need to be handled and stored.
To keep work and memory costs in the Krylov subspace methods restricted,
the so-called restarting can be applied.  One efficient way to restart is 
to use the property of $r_m(t_n + t)$ that its norm is a monotonically
nondecreasing function of time~\cite{ART,BKT21}.  Hence,
we can choose $\tau>0$  such that $\|r_m(t_n + \tau)\|$ is sufficiently
small, advance solution of~\eqref{IVPlin} by $\tau$ and restart
the Krylov method at $t_n+\tau$.
This strategy, called residual-time (RT) restarting, 
leads to an efficient solver for linear problems~\eqref{IVPlin}~\cite{BKT21}.  
However, here we solve a nonlinear
problem and, once a time advance for~\eqref{IVPlin} by $\tau$ is made,
it does not make much sense to proceed in time with the same
linear IVP.  It is more sensible, once the Krylov subspace method 
advanced by~$\tau$ for which the residual $r_m(t_n+ \tau)$ 
is small enough in norm, accept $\tau$ as a new time step size
for the nonlinear IVP~\eqref{IVP}:
$$
t_{n+1}:=t_n+\tau, \quad y^{n+1}:=y_m(t_n + \tau).
$$
This can be seen as a nonlinear variant of the RT restarting.
More precisely, we choose $\tau$ such that
\begin{equation}
\label{phi_stop}
\frac{\|r(t_n+\tau)\|}{\beta}
\leqs \tol_\varphi,
\quad \beta = \|\widehat{g}^n  - \widehat{A}_n y^n\|,
\end{equation}
where $\tol_\varphi$ is a chosen tolerance.
The Krylov subspace method for solving the linear IVP~\eqref{IVPlin}
and finding an appropriate~$\tau$ is outlined in Figure~\ref{f:RT}.
Note that the value \texttt{resnorm} computed in the algorithm is 
$\|r(t_n+\tau)\|/\beta$, the left hand side of 
the inequality~\eqref{phi_stop}. 

\begin{figure}
\begin{center}
\begin{tabular}{rl}
\hline\hline
\multicolumn{2}{l}{Krylov subspace method (based on the Arnoldi process)}\\
\multicolumn{2}{l}{to compute $y(t_n+\tau) = y^n +
  \tau\varphi( -\tau\widehat{A}_n ) 
  \left( \widehat{g}^n  - \widehat{A}_n y^n\right)$}
\\[2ex]
\multicolumn{2}{l}{\textbf{Given:} 
initial step size $\tau_0>0$, 
$\widehat{A}_n=\widehat{A}_n^T\in\Rr^{N\times N}$,
$y^n,\widehat{g}^n\in\Rr^N$,}\\
\multicolumn{2}{l}{\hspace*{7.4ex}tolerance $\tol_\varphi$, 
                   Krylov dimension $\mmax$.}
\\
\multicolumn{2}{l}{\textbf{Output:} 
$\tau>0$ and numerical solution $y^{n+1}=y(t_n+\tau)$ such 
that~\eqref{phi_stop} holds.}
\\\hline
1.  & Initialize zero matrices $H=(h_{ij})\in\Rr^{(\mmax+1)\times \mmax}$ and \\
    & $V\in\Rr^{N\times (\mmax+1)}$ with columns $v_1$, \dots, $v_{\mmax+1}$
\\      
2.  & $v_1 := \widehat{g}^n  - \widehat{A}_n y^n$, $\beta:=\|v_1\|$, 
     $v_1:=v_1/\beta$  \\
3.  & For $j=1,\dots,\mmax$ \\
4.  & \hspace*{2em} $w:=\widehat{A}_n v_j$\\
5.  & \hspace*{2em} For $i=1,\dots,j$\\
6.  & \hspace*{4em} $h_{ij} := w^Tv_i$ \\
7.  & \hspace*{4em} $w := w - h_{ij}v_i$ \\
8.  & \hspace*{2em} end for\\
9.  & \hspace*{2em} $h_{j+1,j} := \|w\|$\\
10. & \hspace*{2em} $e_1:=(1,0,\dots,0)^T\in\Rr^j$, $e_j:=(0,\dots,0,1)^T\in\Rr^j$\\
11. & \hspace*{2em} $s:=\tau_0$, $u:= s\varphi(-s H_{j,j})e_1$\\
12. & \hspace*{2em} $\texttt{resnorm}:=h_{j+1,j}|e_j^Tu|$\\
13. & \hspace*{2em} if $\texttt{resnorm}\leqs \tol_\varphi$\\
14. & \hspace*{4em} convergence: set $\tau:=\tau_0$, break (leave the for loop)\\
15. & \hspace*{2em} elseif $j=\mmax$\\
16. & \hspace*{4em} trace \texttt{resnorm} (lines 11,12) for $s\in[0,\tau_0]$ and set\\
17. & \hspace*{4em} $\tau:=\max\{ s\in[0,\tau_0] \; : \;\; \text{\texttt{resnorm}}\leqs \tol_\varphi\}$\\
18. & \hspace*{4em} $s:=\tau$, $u:= s\varphi(-s H_{j,j})e_1$\\
19. & \hspace*{4em} break (leave the for loop)\\
20. & \hspace*{2em} end if\\
21. & \hspace*{2em} $v_{j+1} := w/h_{j+1,j}$\\
22. & end for\\
23. & Return $\tau$ and $y^{n+1} := y^n + V_j (\beta u)$
\\\hline
\end{tabular}
\end{center}
\caption{Krylov subspace method to make an EE2 time step 
by solving~\eqref{IVPlin}, with~$\widehat{A}_n$ and~$\widehat{g}^n$ defined by
either~\eqref{A_g} or~\eqref{A_g_cs}.
Here $H_{j,j}\in\Rr^{j\times j}$ is formed by the first $j$ rows and columns 
of $H_{m,m}$ and $V_j\in\Rr^{N\times j}$ by the first $j$ columns of $V_m$.}
\label{f:RT}
\end{figure}

This nonlinear RT procedure, where $\tau$ is chosen for a fixed Krylov
subspace dimension~$\mmax$ such that the residual~\eqref{rm} is
small enough in norm, is well suited for an adaptive time stepping.
If, for a certain reason, the Cahn--Hilliard system~\eqref{IVP} 
has to be integrated with a constant time step size~$\tau$ then we can
apply the standard linear RT restarting, see~\cite{ART,BKT21} and 
its implementation available 
at \url{https://team.kiam.ru/botchev/expm/}.
For matrices~$\widehat{A}_n$ whose numerical range has elements
with nonnegative real part, the RT restarting guarantees that~\eqref{IVPlin} 
is solved on the whole time interval 
$[t_n,t_n+\tau]$ for any~$\tau$ and~$\mmax$.
However, for the Cahn--Hilliard equation the matrices~$\widehat{A}_n$ 
may have numerical range with small negative elements, leading to
an accuracy corruption in the RT restarting.  Therefore, in a constant time step
setting, the Krylov subspace dimension $\mmax$ should not be taken too small 
and $\tau$ should not be taken too large.   

\subsection{Accuracy order of EE2}
\label{s:accEE}
The following result shows that the EE2 time integration scheme has 
a nonstiff accuracy order~2.

\begin{proposition}
\label{p:ord2}
Let the EE2~\eqref{IVPlin},\eqref{EE2} scheme be applied  
to solve the space discretized Cahn--Hilliard IVP problem~\eqref{IVP}
and let the exact solution~$y(t)$ of IVP~\eqref{IVP} be Lipschitz 
continuous.
Then for the local error $e^n$ of the EE2 scheme at time step~$n$ holds
\begin{equation}
\label{ord2}
e^n = \Oo (\tau^3),  
\end{equation}
where the leading term is proportional to~$\|A\|$, with $-A$ 
being the discretized Laplacian defined in~\eqref{IVP}. 
\end{proposition}

\begin{proof}
Let $y(t)$ and $\tilde{y}(t)$ be solutions of~\eqref{IVP} and~\eqref{IVPlin},
respectively.
To get an expression for the local error $e^{n+1}$ we assume, by
definition of the local error, that solution $y^n$ at time step~$n$
is exact, $y^n=y(t_n)$.  Let $e(t)\equiv y(t)-\tilde{y}(t)$, $t\geqs t_n$, 
so that $e^{n+1}=e(t_{n+1})$.  Then we have
\begin{align*}
e'(t) &= y'(t)-\tilde{y}'(t) = -AF'(y) - \epsilon^2A^2 y(t) 
+\widehat{A}_n \tilde{y}(t) - \widehat{g}^n\\
&= -AF'(y) - \epsilon^2A^2 y(t) 
+AJ_n\tilde{y}(t) +\epsilon^2A^2\tilde{y}(t) - AJ_ny^n +AF'(y^n)\\
&= -\epsilon^2A^2 e(t) -A\left[ F'(y)-F'(y^n) - J_n(\tilde{y}(t)-y^n)\right]\\
&= -\epsilon^2A^2 e(t) -A\left[ F'(y)-F'(y^n) - 
   J_n(\tilde{y}(t)-y(t)+y(t)-y^n)\right]\\
&= -\epsilon^2A^2 e(t) -AJ_n(y(t)-\tilde{y}(t)) -A\left[ F'(y)-F'(y^n) - 
   J_n(y(t)-y^n)\right]\\
&= -A (J_n + \epsilon^2A) e(t)-A\left[ F'(y)-F'(y^n) - J_n(y(t)-y^n)\right], \\
\end{align*}
so that $e(t)$ satisfies IVP
$$
e'(t)= -\widehat{A}_n e(t) + g(y(t)),
\quad t\geqs t_n, \quad
e(t_n) = 0,
$$
with
$$
g(y(t))=-A\left[ F'(y(t))-F'(y(t_n)) - J_n(y(t)-y(t_n))\right].
$$
Applying the variation-of-constants formula (see,
e.g.,~\cite[Section~I.2.3]{HundsdorferVerwer:book} or
\cite[relation~(1.5)]{HochbruckOstermann2010})
\begin{equation*}
\label{VOC}
e(t) = \exp(-t\widehat{A}_n)e(t_n) + \int_{t_n}^t \exp(-(t-s)\widehat{A}_n)
g(y(s)) d s,
\quad t\geqs t_n,
\end{equation*}
with $\exp$ being the matrix exponential,
and taking into account that $e(t_n)=0$, we can write, 
for $t_n\leqs t\leqs t_n+\tau$,
\begin{align*}
\|e(t)\| &\leqs \max_{s\in[t_n,t]}\| g(y(s))\|
\int_{t_n}^t \|\exp(-(t-s)\widehat{A}_n) \|d s
\leqs 
\max_{s\in[t_n,t]}\| g(y(s))\|\int_{t_n}^t e^{-\omega(t-s)} d s
\\ 
&= \max_{s\in[t_n,t]}\| g(y(s))\| (t-t_n)\varphi(-\omega (t-t_n))
\leqs \max_{s\in[t_n,t]}\| g(y(s))\| \tau\varphi(-\omega \tau),
\end{align*}
where $\omega\in\Rr$ is a constant such that for $t\geqs 0$
it holds $\|\exp(-t\widehat{A}_n)\|\leqs e^{-\omega t}$, 
see, e.g.,~\cite[Theorem~I.2.4]{HundsdorferVerwer:book}
or~\cite[relation~(2.6)]{HochbruckOstermann2010}.

Note that 
$$
F'(y(t))-F'(y(t_n)) - J_n(y(t)-y(t_n)) = D_n^{(2)} ((y(t)-y(t_n))^2
+\Oo(\|y(t)-y(t_n)\|^3),
$$
where $D_n^{(2)}$ is a diagonal matrix with the entries
$6y_i^n$, $i=1,\dots,N$, the second derivative values of $F$ at $y^n$,
and squaring in $((y(t)-y(t_n))^2$ is understood elementwise.
Therefore, taking into account that $y(t)$ is Lipschitz
continuous and, hence, $\Oo(\|y(t)-y(t_n)\|^2=\Oo(\tau^2)$, 
we obtain 
$$
g(y(s)) = \|A\| \Oo(\tau^2), \quad s\in[t_n,t_n+\tau].
$$ 
Substituting the last relation into the estimate for~$\|e(t)\|$ above
and taking into account that $e^{n+1}=e(t_{n+1})$ concludes the proof.
\end{proof}

Thus, the EE2 scheme has accuracy order~2, one lower than the
local error.
Since $-A$ is the discretized Laplacian, we may expect   
$\|A\|\sim h^{-2}$.  Therefore, as the leading term in the 
estimate~\eqref{ord2} is proportional to~$\|A\|$, an order
reduction for fine space grids may occur 
(see, e.g.,~\cite{HundsdorferVerwer:book} or~\cite{ODE2(HW)}).

\subsection{The EE2 scheme combined with the Eyre convex splitting}
\label{s:EEcs}
A combination of the EE2 scheme with the Eyre convex splitting can be 
obtained by considering a time continuous analogue of the 
NLSS scheme~\eqref{EBcs},
\begin{equation}
y'(t) = -A 
\left(F'(y(t)) + \epsilon^2 Ay(t) + y(t) - y^n \right),
\end{equation}
and applying linearization~\eqref{linEE}.  This results in
an Eyre-stabilized EE2 scheme~\eqref{IVPlin},\eqref{EE2}, where 
$\widehat{A}_n$ and $\widehat{g}^n$ are now given by~\eqref{A_g_cs}.

\subsection{Adaptive time step selection in the EE2 scheme}
\label{s:EEadapt}
An aposteriori time step size selection in the EE2 scheme can
be realized in the same way as for the LIM scheme.  More precisely,
once the next time step solution $y^{n+1}$ is computed by~EE2,
we substitute $y^{n+1}$ to the right hand side of the expression
for~$y_{\text{PC}}^{n+1}$ in~\eqref{LIMest} and compute $\texttt{est}_{n+1}$.
Since the EE2 scheme is second order accurate, a time step size for 
the next step is then chosen as
\begin{equation}
\label{EEdt} 
\tau_{\text{new}} := \min\left\{ \frac54 \tau,
\sqrt{\frac{\tol}{\texttt{est}_{n+1}}}\;\tau \right\},
\end{equation}
where $\tol$ is a prescribed accuracy tolerance.
The resulting adaptive EE2 time stepping scheme is presented 
in Figure~\ref{f:EEadapt}. 

\begin{figure}
\begin{center}
\begin{tabular}{rl}
\hline\hline
\multicolumn{2}{l}{The EE2 adaptive time stepping scheme for Cahn--Hilliard equation~\eqref{IVP}}
\\[2ex]
\multicolumn{2}{l}{\textbf{Given:} 
initial step size $\tau_0>0$, final time $T$,
$A\in\Rr^{N\times N}$, $y^0\in\Rr^N$,}\\
\multicolumn{2}{l}{$\phantom{\textbf{Given:}}$ 
tolerance~\tol, Krylov dimension~$\mmax$}
\\
\multicolumn{2}{l}{\textbf{Output:} 
numerical solution $y_{\text{fin}}$ at time~$T$.}
\\\hline
1.  & Set $t:=0$, $\tau:=\tau_0$, $y^n:=y^0$ \\
2.  & While $t<T$ do \\
3.  & \hspace*{2em} if $t+\tau > T$ \\
4.  & \hspace*{4em}    $\tau:= T-t$ \\
5.  & \hspace*{2em} end if          \\
6.  & \hspace*{2em} Compute $\widehat{A}_n$ and $\widehat{g}^n$ 
                    by relation~\eqref{A_g} or, if the Eyre splitting is used,
                    by~\eqref{A_g_cs}\\
7.  & \hspace*{2em} With $\tol_{\varphi}$, set by~\eqref{tol_phi}, 
                    and $\mmax$, \\
    & \hspace*{2em} apply Krylov subspace method (Figure~\ref{f:RT})
                    to obtain $\tau$ and $y^{n+1}$\\
8.  & \hspace*{2em} Compute $y_{\text{PC}}^{n+1}$ and $\texttt{est}_{n+1}$ 
                    by relation~\eqref{LIMest}\\
9.  & \hspace*{2em} Set $y^{n}:= y^{n+1}$, $t:=t + \tau$ \\
10. & \hspace*{2em} Set $\tau:=\tau_{\text{new}}$ by relation~\eqref{EEdt}\\
11. & end while\\ 
12. & Return solution $y_{\text{fin}}:=y^{n+1}$
\\\hline
\end{tabular}
\end{center}
\caption{The EE2 adaptive time stepping scheme}
\label{f:EEadapt}
\end{figure}

The tolerance $\tol_{\varphi}$ for the Krylov subspace method is chosen
at each time step~$n$ as
\begin{equation}
\label{tol_phi}  
\tol_{\varphi} := 
\max\left\{ 
\min \left\{ 
\frac{\|\widehat{g}^n\|}{10\beta} \,,\, 
\frac1{10} \,,\, 10\,\tol \right\}\; , \; 10^{-7} 
\right\},
\end{equation}
with $\beta =\|\widehat{g}^n  - \widehat{A}_n y^n\|$ and
$\tol$ being the prescribed tolerance.
Taking into account~\eqref{phi_stop}, we see that the terms 
$\|\widehat{g}^n\|/(10\beta)$ and $1/10$ here guarantee that 
the residual norm (cf.~\eqref{rm}) is an order smaller than 
$\|\widehat{g}^n\|$ and the right side norm 
$\beta$ of the linearized IVP~\eqref{IVPlin} being solved, respectively.
In addition, the term $10\,\tol$ links the Krylov subspace tolerance
$\tol_{\varphi}$ to the accuracy tolerance of the scheme~$\tol$
and the term $10^{-7}$ avoids a too stringent tolerance.

\section{Numerical experiments}
\label{s:num}
\subsection{Test setting}
The test runs described in this section are carried out in 
Octave, version 8.3.0, on a Linux PC with twenty 3.30GHz~CPUs
and 64 Gb memory.  
The test problem is adopted from~\cite[Section~3.1]{Lee_ea2013},
it is a 2D Cahn--Hilliard equation~\eqref{CH} posed in domain
$\Omega=(0,64)\times (0,64)$. The standard
second order finite difference space discretization on a uniform 
$n_x\times n_y$ grid, with nodes $(x_i,y_j)$, $x_i=64(i-1/2)/n_x$, 
$y_j=64(j-1/2)/n_y$, leads to IVP~\eqref{IVP} 
of size $N=n_xn_y$.
The problem is integrated in time for $t\in[0,T]$ with $T=1000$.
Initial solution vector~$y^0$ is set to a vector with random entries
uniformly distributed on the interval~$(-0.01,0.01)$.
In all tests the value of~$\epsilon$ is taken to be $\epsilon_4$
for the grid~$64\times 64$, see relation~\eqref{eps}. 
As discussed above, this means that
$\|\widehat{A}_n\|=\Oo(h^{-4})$ and the problem~\eqref{IVP} is stiff.

We test accuracy of the LIM and EE2 schemes by comparing their
solution $y^n$ at the final time~$t_n=T$ to a reference
solution $y_{\text{ref}}^n$ computed on the same space grid with 
a very high time accuracy.
More specifically, the error values reported in this section are
relative error norms
\begin{equation}
\label{acc_reached}  
\text{error} =\frac{\|y^n - y_{\text{ref}}^n\|}{\|y_{\text{ref}}^n\|}.
\end{equation}
Since the reference solution is computed on the same space grid,
error measured in this way can be seen as an adequate indicator
of the time error~\cite{RKC97}.
We compare computational costs in the LIM and EE2 schemes in terms of 
the number matrix-vector multiplications with the 
matrix~$\widehat{A}_n$ (cf.~\eqref{A_g},\eqref{A_g_cs}).  
This is justified provided the Krylov subspace dimension $\mmax$
is kept moderate, so that, taking into account that $\mmax\ll N$, 
the overhead in the EE2 scheme to compute the matrix~$H_{m,m}$ and 
the matrix function~$\varphi(-\tau H_{m,m})$ 
is negligible (see the algorithm in Figure~\ref{f:RT}).

\begin{figure}
\includegraphics[width=0.49\linewidth]{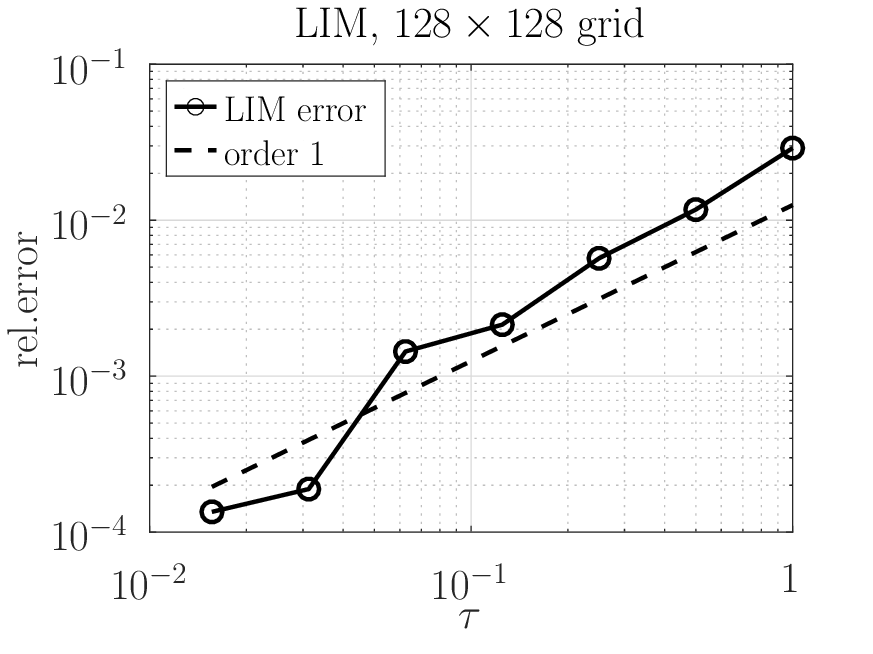}%
\includegraphics[width=0.49\linewidth]{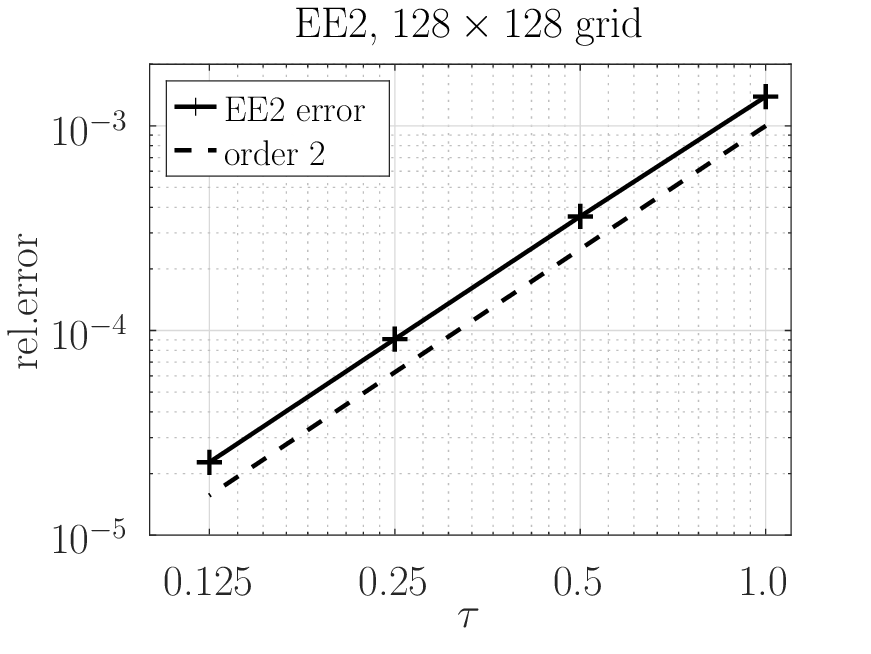}
\caption{Error convergence plots for the LIM (left) and the EE2 (right) 
schemes run with constant time step size~$\tau$ on the 
$128\times 128$ space grid.  The plots are produced with the 
data from Table~\ref{t:dt_const}.}
\label{f:conv}
\end{figure}

\subsection{The LIM and EE2 schemes with a constant time step size}
We first compare the two time integrators run with a constant time step
size.  
It turns out that, for large constant time step sizes, small negative 
eigenvalues of~$\widehat{A}_n$ may have a negative effect on the 
RT restarting procedure in the Krylov subspace methods.
This is because the RT restarting procedure is guaranteed
to work for matrices whose numerical range does not contain
elements with negative real part (which is a typical assumption in 
analysis of Krylov subspace methods)~\cite{ART,BKT21}.  
This problem can be solved 
by either reducing
the time step size at an initial evolution phase or by using
a large Krylov subspace dimension.  For simplicity of the 
experiments, we choose to set the Krylov subspace dimension
in the EE2 time integrator to a large value $\mmax=100$.
We emphasize that this is done only if EE2 is used with 
a constant time step size.  With the adaptive time stepping
strategy we take either $\mmax=10$ or $\mmax=30$.
As already discussed, taking the large value $\mmax=100$ makes 
comparing costs of both schemes in terms of the matrix-vector products 
with~$\widehat{A}_n$ unfair.  Hence, these values are given for indication
only.

\begin{table}
\caption{Achieved accuracy and the number 
of spent matvecs (matrix-vector products) for the LIM and the EE2
schemes run with constant time step size~$\tau$.  The EE2 is not 
run with $\tau=0.0625$ because a high accuracy of convergence order~2
is observed for~$\tau\geqs 0.125$.}
\label{t:dt_const}
\begin{center}
\begin{tabular}{ccccc}
\hline\hline
$\tau$  & \multicolumn{2}{c}{LIM}    & \multicolumn{2}{c}{EE2} \\
        & \# matvecs    &error\ \ \  &\ \ \ \# matvecs    &  error (\tol) \\
\hline
\multicolumn{5}{c}{$64\times 64$ grid}\\
1.0     &   12960 & {\tt9.60e-02} & 10217 & {\tt3.24e-04} ($10^{-3}$)\\
0.5     &   18000 & {\tt1.10e-03} & 24824 & {\tt1.55e-04} ($10^{-5}$)\\ 
0.25    &   27856 & {\tt2.52e-03} & 39138 & {\tt3.88e-05} ($10^{-6}$)\\
0.125   &   40000 & {\tt8.34e-04} & 63039 & {\tt9.52e-06} ($10^{-7}$)\\
0.0625  &   48000 & {\tt3.20e-04} &  ---  &   --- \\
\hline
\multicolumn{5}{c}{$128\times 128$ grid}\\
1.0     &   50770 & {\tt2.90e-02} & 41266  & {\tt1.60e-03} ($10^{-3}$)\\
0.5     &   69908 & {\tt1.17e-02} & 92650  & {\tt3.61e-04} ($10^{-5}$)\\
0.25    &   99824 & {\tt5.70e-03} & 141148 & {\tt9.08e-05} ($10^{-6}$)\\
0.125   &  136000 & {\tt2.14e-03} & 210401 & {\tt2.27e-05} ($10^{-7}$)\\
0.0625  &  207350 & {\tt1.44e-03} &  ---   &   --- \\
\hline
\end{tabular}
\end{center}
\end{table}
 
To run with a constant $\tau$, the EE2 scheme needs a tolerance 
parameter~$\tol_{\varphi}$ for the Krylov subspace evaluation
of the $\varphi$ matrix function.  Therefore, we supply EE2
with a tolerance value $\tol$, for which~$\tol_{\varphi}$ gets
its value according to~\eqref{tol_phi}.

First, in Table~\ref{t:dt_const} for both schemes we give values of the achieved
accuracy, measured as in~\eqref{acc_reached}, and required
number of matrix-vector products.  Then, in Figure~\ref{f:conv}
error convergence plots for both schemes are presented.
As we see, the EE2 scheme has indeed accuracy order~2 provided
the linearized IVP~\eqref{IVPlin} is solved at each time step
sufficiently accurately
(to achieve this, we decrease the tolerance value~\tol{} with
the time step size~$\tau$).  The LIM scheme shows, as expected,
at least an order~1 accuracy.

\begin{figure}
\includegraphics[width=0.49\linewidth]{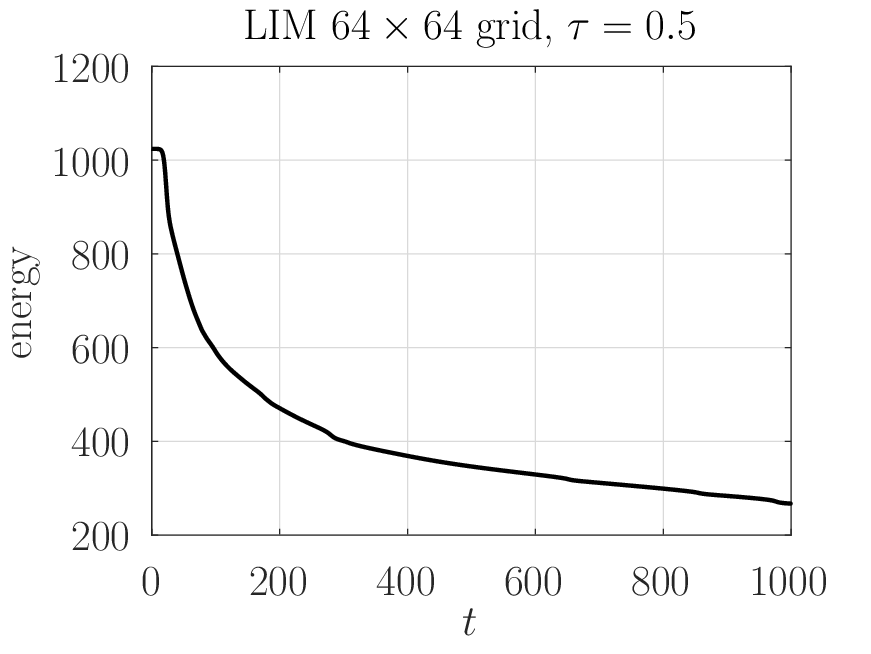}%
\includegraphics[width=0.49\linewidth]{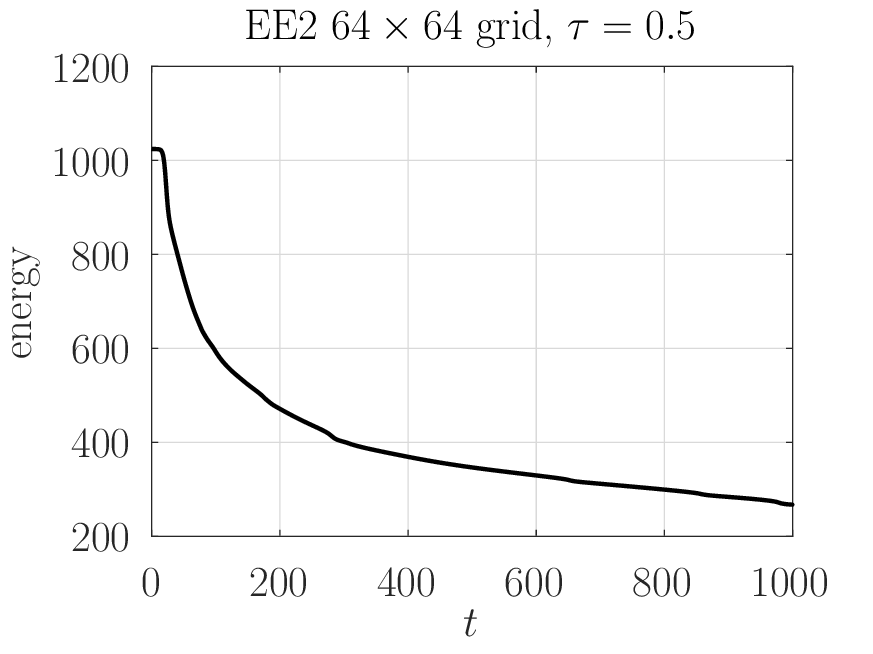}
\caption{Discrete energy~\eqref{Ediscr} versus time for the LIM (left) 
and EE2($\tol=10^{-2}$) (right) schemes run with constant $\tau=0.5$}
\label{f:Ee1}  
\end{figure}

In Figure~\ref{f:Ee1} for both schemes we show time evolution plots
of discrete energy, computed on a $n_x\times n_y$ grid 
at each time step~$n$ for $y^n=y$ as~\cite{Lee_ea2013} 
\begin{equation}
\label{Ediscr}  
\Ee_h(y) = h_xh_y\sum_{i,j=1}^{n_x,n_y} F(y_{i,j}) + \frac{\epsilon^2}2
\left(
\sum_{i=0,j=1}^{n_x,n_y}(y_{i+1,j}-y_{i,j})^2 + 
\sum_{i=1,j=0}^{n_x,n_y}(y_{i,j+1}-y_{i,j})^2
\right),
\end{equation}
with $h_{x,y}$ being the space grid sizes in the 
$x$ (respectively, $y$) direction. 
In addition, Figure~\ref{f:m1} presents plots of mass deviation
versus time in both schemes.  To produce these figures, we intentionally
take a large time step size~$\tau=0.5$ and, for~EE2, a relaxed tolerance 
value~$\tol=10^{-2}$: as we see, even in this case both schemes exhibit
good conservation properties.

\begin{figure}
\includegraphics[width=0.49\linewidth]{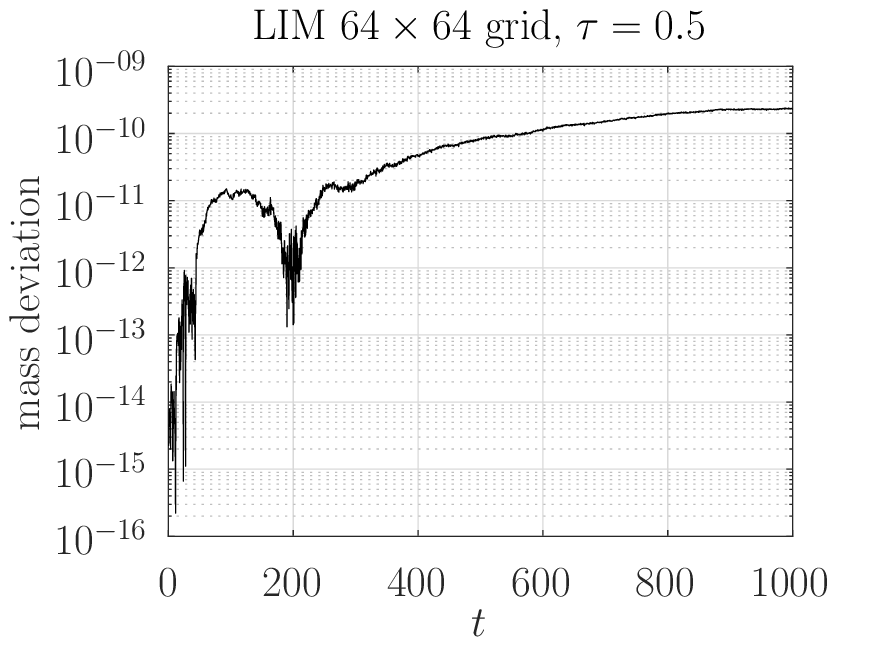}
\includegraphics[width=0.49\linewidth]{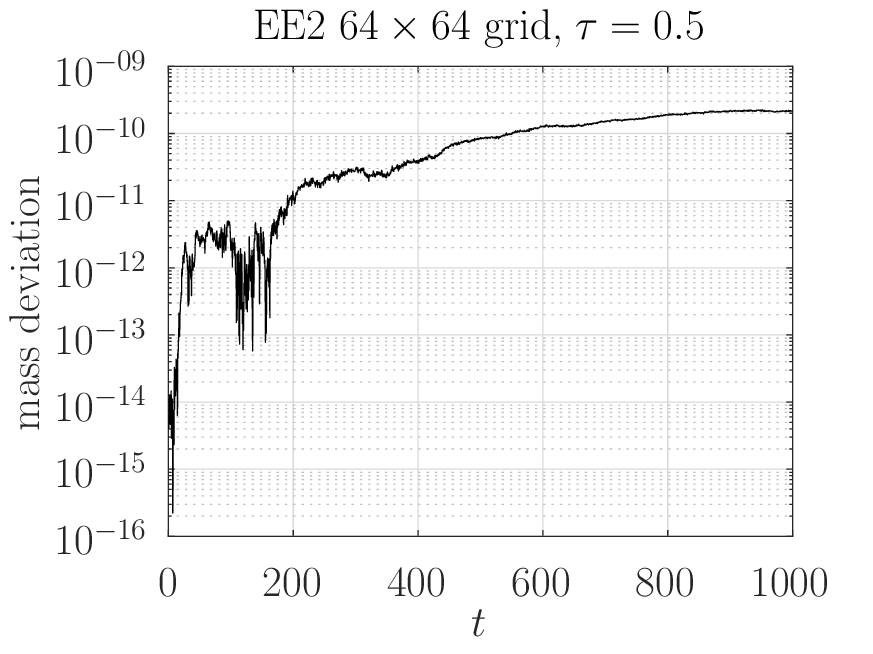}
\caption{Mass deviation $|(m^n-m^0)/m^0-1|$, with $m^n=\sum_jy_j^n$, 
versus time for the LIM (left) and EE2($\tol=10^{-2}$) (right) schemes 
run with constant $\tau=0.5$}
\label{f:m1}  
\end{figure}

\subsection{LIM and EE2 schemes with the Eyre splitting}
We now present test results for the LIM and EE2 schemes combined
with the Eyre splitting.  
Just as in the previous test, the EE2 scheme is provided with
a tolerance parameter~$\tol$, for which the tolerance in the Krylov
subspace method~$\tol_{\varphi}$ is determined by formula~\eqref{tol_phi}. 
The matrix~$\widehat{A}_n$ is now similar to a symmetric positive
semidefinite matrix and, hence, is guaranteed to
have nonnegative eigenvalues (see Proposition~\ref{p:stab}), so that
the RT restarting within the Krylov subspace evaluation
of the matrix function~$\varphi$ should now work reliably.
Therefore, we set the maximal Krylov subspace dimension 
$\mmax$ to a realistic value~$\mmax=30$ and denote the scheme
with this $\mmax$ value as EE2(30).  Choosing such a value of~$\mmax$ 
allows to compare the costs in two schemes in terms of the needed 
matrix-vector multiplication number.
 
Error values achieved by both schemes, together with the number
of required matrix-vector multiplications, are presented in
Table~\ref{t:dt_Eyre}.  For the EE2(30) scheme, these data are obtained
with a relaxed tolerance value~$\tol=10^{-2}$ which, taking into account
the large error triggered by the Eyre splitting, appears to be sufficient.  
In the left plot of Figure~\ref{f:Eyre}, we plot the error values for both
schemes versus the time step size~$\tau$.  The plot is produced
with smaller tolerance values in the EE2 scheme, for which
the error values have converged (i.e., decreasing tolerance further 
does not change the error value). The plot
confirms a well known fact that the Eyre splitting has accuracy 
order~1.  
In the right plot of Figure~\ref{f:Eyre},
the error values reported in Table~\ref{t:dt_Eyre} are given
versus the required matrix-vector multiplication number. 

As we see, the EE2(30) scheme allows to achieve an accuracy comparable
to that of the LIM scheme for much smaller number of matrix-vector
multiplications.  The price for this increased efficiency is storage
and handling of $\mmax=30$ Krylov subspace vectors.  
As is clear from the results in the previous sections, the LIM and EE2
schemes have essentially different error behavior.  
Therefore, almost indistinguishable error convergence curves 
for the Eyre-stabilized LIM and EE2 schemes 
(see the left plot in Figure~\ref{f:Eyre}) are probably due to
a large Eyre splitting error which dominates the errors in both
schemes.    
Discrete energy, given by~\eqref{Ediscr}, and mass deviation plots 
for the EE2 scheme combined 
the Eyre splitting are presented in Figure~\ref{f:Eyre2}.

\begin{table}
\caption{Achieved accuracy and the number 
of spent matvecs (matrix-vector products) for the LIM and the EE2($\mmax=30$)
schemes combined with the Eyre splitting and run with constant
time step size~$\tau$.}
\label{t:dt_Eyre}
\begin{center}
\begin{tabular}{ccccc}
\hline\hline
$\tau$  & \multicolumn{2}{c}{LIM}    & \multicolumn{2}{c}{EE2(30)} \\
        & \# matvecs    &error\ \ \  &\ \ \ \# matvecs    &  error (\tol) \\
\hline
\multicolumn{5}{c}{$64\times 64$ grid}\\
0.5     &  21878  & {\tt9.93e-01} & 13776 & {\tt4.84e-01} ($10^{-2}$)\\ 
0.25    &  28000  & {\tt6.34e-01} & 17391 & {\tt5.33e-01} ($10^{-2}$)\\
0.125   &  40000  & {\tt5.05e-01} & 16965 & {\tt6.67e-02} ($10^{-2}$)\\
0.0625  &  48000  & {\tt7.47e-02} & 23910 & {\tt2.41e-02} ($10^{-2}$)  \\
0.03125 &  96000  & {\tt3.85e-02} & 35818 & {\tt1.13e-02} ($10^{-2}$)  \\
\hline
\multicolumn{5}{c}{$128\times 128$ grid}\\
0.5     &   70000 & {\tt8.02e-01} &  54973 & {\tt7.98e-01} ($10^{-2}$)\\
0.25    &  100000 & {\tt5.65e-01} &  66443 & {\tt4.58e-02} ($10^{-2}$)\\
0.125   &  150444 & {\tt3.06e-02} &  78237 & {\tt2.92e-02} ($10^{-2}$)\\
0.0625  &  208000 & {\tt1.84e-02} &  76857 & {\tt2.16e-02} ($10^{-2}$)\\
0.03125 &  288000 & {\tt1.03e-02} & 110008 & {\tt1.15e-02} ($10^{-2}$)\\
\hline
\end{tabular}
\end{center}
\end{table}

\begin{figure}
\includegraphics[width=0.49\linewidth]{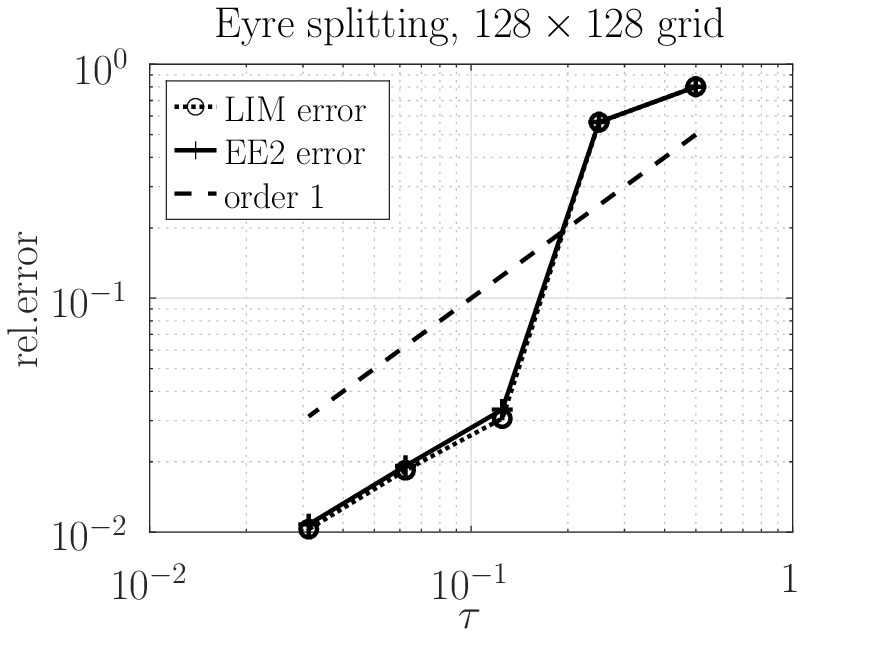}
\includegraphics[width=0.49\linewidth]{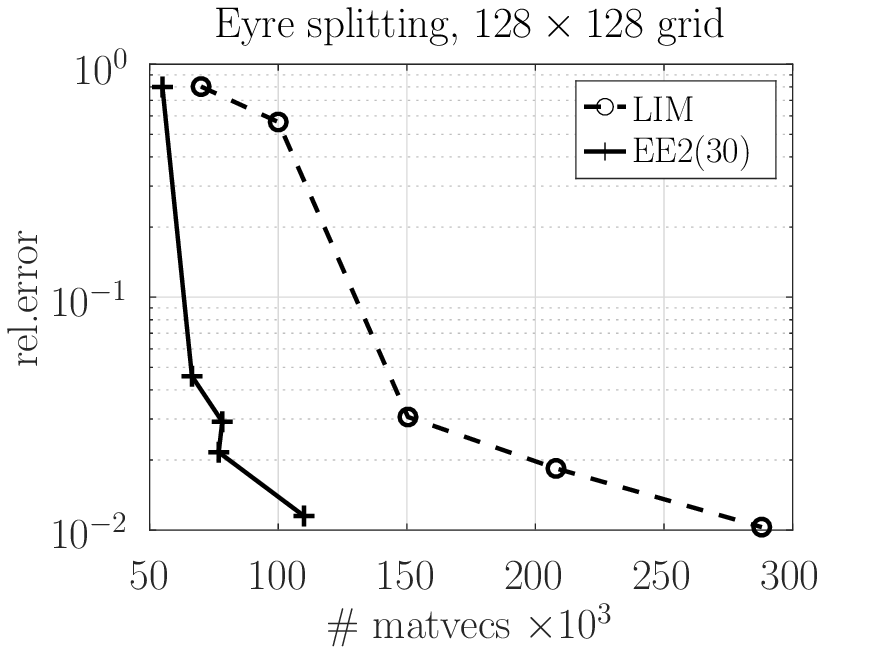}
\caption{Left: error convergence rate  
for the LIM and EE2(30) combined with the Eyre splitting. 
Right: achieved error versus matrix-vector multiplication number
for the LIM and EE2(30) combined with the Eyre splitting. 
The right plot is produced with the data from Table~\ref{t:dt_Eyre}.}
\label{f:Eyre}  
\end{figure}

\begin{figure}
\includegraphics[width=0.49\linewidth]{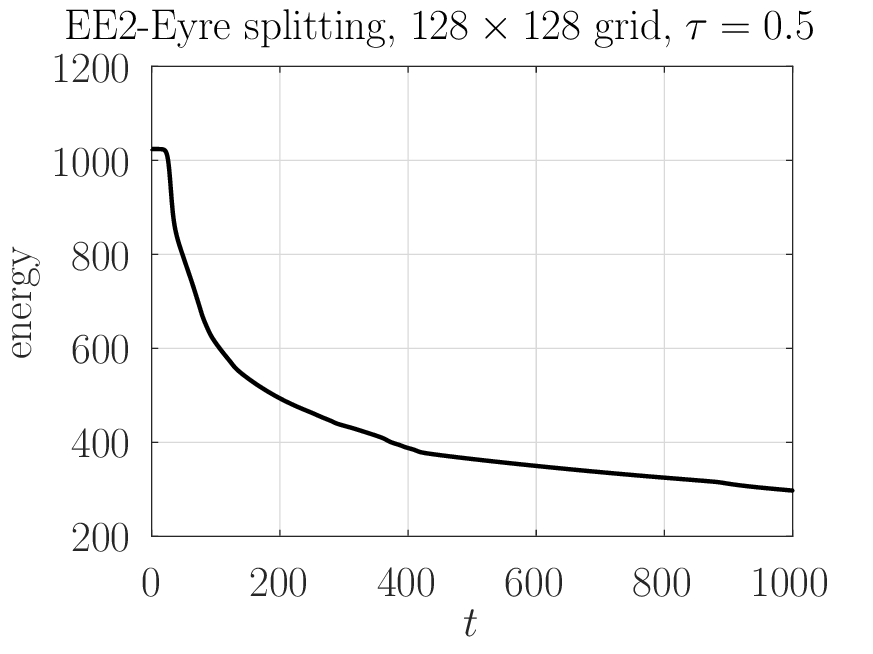}
\includegraphics[width=0.49\linewidth]{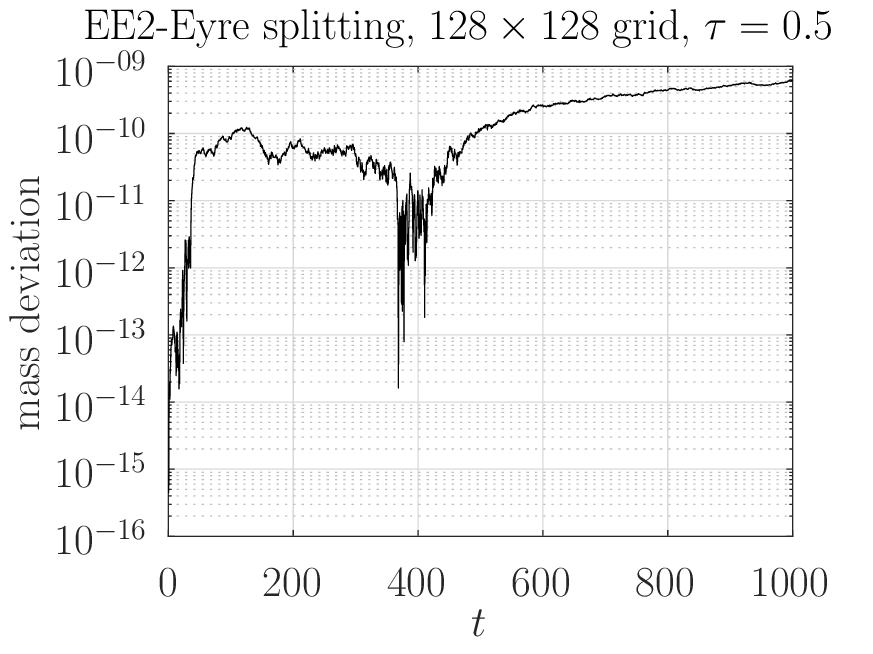}  
\caption{Discrete energy~\eqref{Ediscr} (left) and mass deviation
 $|(m^n-m^0)/m^0-1|$, with $m^n=\sum_jy_j^n$, (right) 
versus time for the EE2($\mmax=30$) scheme combined with the Eyre
splitting and run with constant $\tau=0.5$ and $\tol=10^{-2}$.}
\label{f:Eyre2}    
\end{figure}

\subsection{The LIM and EE2 schemes with adaptively chosen time step size}
In this section both schemes are tested with time step size chosen
adaptively as discussed in Sections~\ref{s:LIMadapt} and~\ref{s:EEadapt}.
In this setting, a tolerance value~$\tol$ is provided to both schemes
which is then used to choose an appropriate time step size 
(cf.~\eqref{LIMdt},\eqref{EEdt}) and, in 
the EE2 scheme, also to determine the tolerance value~$\tol_{\varphi}$
for the Krylov subspace method in~\eqref{tol_phi}.
In addition, both schemes are given an initial time step size~$\tau_0$.
In the EE2 scheme the value of~$\tau_0$ can be instantly decreased 
by the Krylov subspace method while the time step is actually made, 
see Figure~\ref{f:RT}.  That is why in the EE2 scheme we set $\tau_0$ 
to a large value $\tau_0=1$ in all the runs.
The $\tau_0$ values used in the LIM scheme are reported. 
The maximal Krylov subspace dimension~$\mmax$ in the EE2 scheme is 
set to either~10 or~30 and we denote the resulting scheme 
as EE2(10) and EE2(30), respectively.   

\begin{table}
\caption{Achieved accuracy and the number of spent matvecs (matrix-vector 
products) for the LIM and the EE2 schemes run with adaptively chosen~$\tau$
on the $128 \times 128$ space grid.  For the EE2 scheme
initial time step size value is initially set to $\tau_0=1$ 
and instantly changed by the Krylov subspace method, see Figure~\ref{f:RT}.
The maximal Krylov subspace dimension is $\mmax=10$ in~EE2(10)
and $\mmax=30$ in EE2(30).}
\label{t:var_dt128}
\begin{center}
\begin{tabular}{ccc}
%
\hline\hline
method, \tol{} ($\tau_0$)  & \# matvecs    &  error  \\
\hline
LIM, $10^{-2}$       ($\tau_0=1.0$)  &   38940 & {\tt2.73e-03} \\
LIM, $10^{-3}$       ($\tau_0=0.5$)  &  120476 & {\tt3.62e-04} \\
LIM, $10^{-4}$       ($\tau_0=0.25$) &  397133 & {\tt1.22e-04} \\
EE2(10), $10^{-2}$                   &  42856  & {\tt9.63e-03} \\
EE2(10), $10^{-3}$                   &  57100  & {\tt3.31e-03} \\
EE2(10), $10^{-4}$                   & 110088  & {\tt7.57e-05} \\
EE2(30), $10^{-2}$                   &  22800  & {\tt4.62e-03} \\
EE2(30), $10^{-3}$                   &  44493  & {\tt7.03e-04} \\
EE2(30), $10^{-4}$                   &  73867  & {\tt1.28e-04} \\
\hline
\end{tabular}
\end{center}
\end{table}

\begin{table}
\caption{Achieved accuracy and the number of spent matvecs (matrix-vector 
products) for the LIM and the EE2 schemes run with adaptively chosen~$\tau$
on the $192 \times 192$ space grid.  For the EE2 scheme
initial time step size value is initially set to $\tau_0=1$ 
and instantly changed by the Krylov subspace method, see Figure~\ref{f:RT}.
The maximal Krylov subspace dimension is $\mmax=10$ in~EE2(10)
and $\mmax=30$ in EE2(30).}
\label{t:var_dt192}
\begin{center}
\begin{tabular}{ccc}
\hline\hline
method, \tol{} ($\tau_0$)  & \# matvecs    &  error  \\
\hline
LIM, $10^{-2}$       ($\tau_0=1.0$)  &   85843 & {\tt1.12e-03} \\ 
LIM, $10^{-3}$       ($\tau_0=0.5$)  &  265411 & {\tt1.82e-04} \\ 
LIM, $10^{-4}$       ($\tau_0=0.25$) &  893808 & {\tt3.79e-05} \\ 
EE2(10), $10^{-2}$                   &  201460 & {\tt1.07e-03} \\ 
EE2(10), $10^{-3}$                   &  208018 & {\tt1.04e-03} \\ 
EE2(10), $10^{-4}$                   &  339418 & {\tt4.33e-04} \\ 
EE2(10), $10^{-5}$                   &  546624 & {\tt5.35e-06} \\ 
EE2(30), $10^{-2}$                   &   75053 & {\tt2.99e-03} \\ 
EE2(30), $10^{-3}$                   &  126485 & {\tt5.35e-04} \\ 
EE2(30), $10^{-4}$                   &  221243 & {\tt1.21e-05} \\ 
\hline
\end{tabular}
\end{center}
\end{table}

In Tables~\ref{t:var_dt128} and~\ref{t:var_dt192}, for the space 
grids~$128\times 128$ and~$192\times 192$, respectively,
we show achieved error values and corresponding required matrix-vector
multiplication numbers.  The data shown in the tables are also
visualized in the error-versus-work plots in Figure~\ref{f:acc_vs_mv}.
As we see, for moderate accuracy requirements the LIM scheme is slightly
more efficient than the EE2(10) scheme.  For stricter tolerance values
the EE2(10) becomes more efficient than LIM.  The EE2(30) outperforms
both the LIM and EE2(10) schemes for the whole tolerance range.

\begin{figure}
\includegraphics[width=0.49\linewidth]{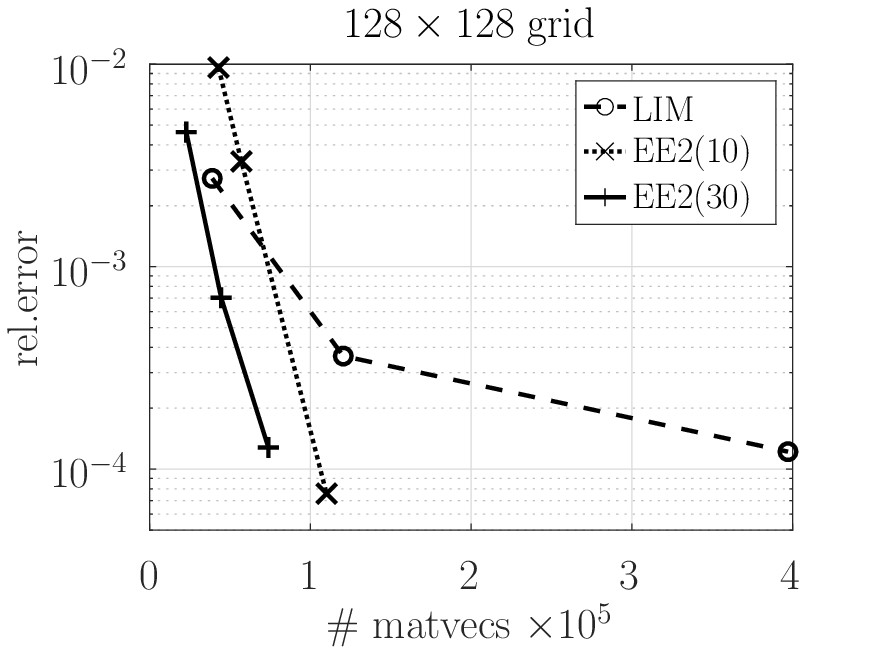}
\includegraphics[width=0.49\linewidth]{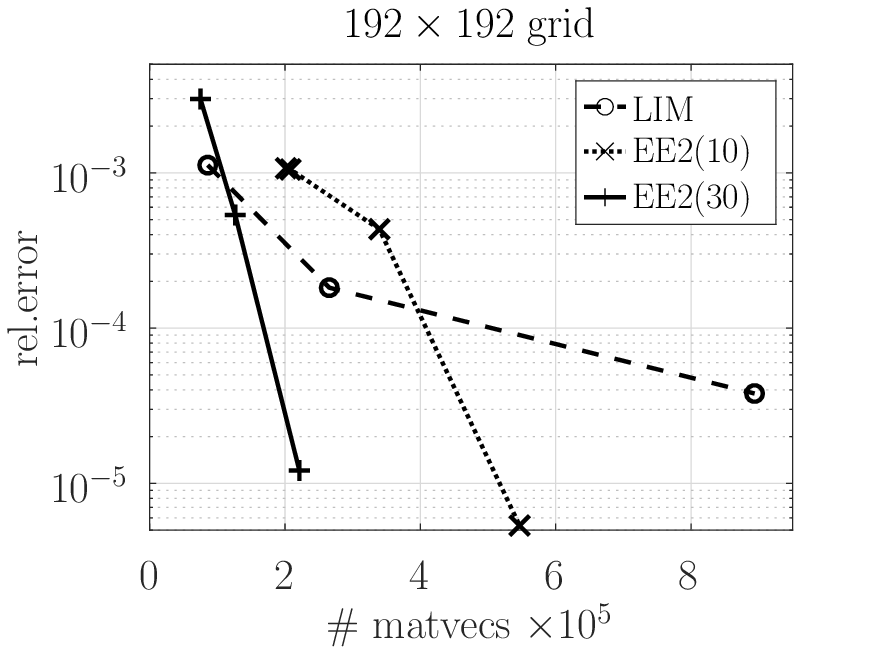}
\caption{Achieved accuracy versus number of required matrix-vector products 
(matvecs) for the adaptive LIM, EE2(10) and EE2(30) schemes
for $128\times 128$ (left) and $192\times 192$ (right) grids}
\label{f:acc_vs_mv}
\end{figure}

In Figure~\ref{f:Ee2} plots of the total energy versus time are presented
for the adaptive LIM($\tau_0=1$) and EE2(10) schemes, where the discrete energy is computed according to~\eqref{Ediscr}.
We deliberately produce these plots
for both schemes run with a tolerance value~$10^{-2}$:
as we see, even with this moderate tolerance value the energy behavior
in both schemes is adequate.

\begin{figure}
\includegraphics[width=0.49\linewidth]{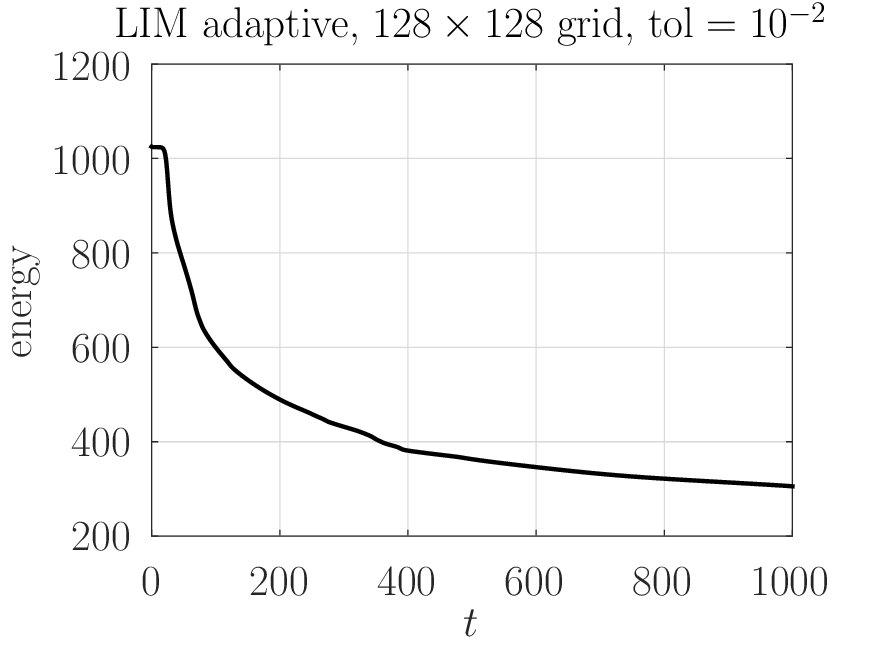}
\includegraphics[width=0.49\linewidth]{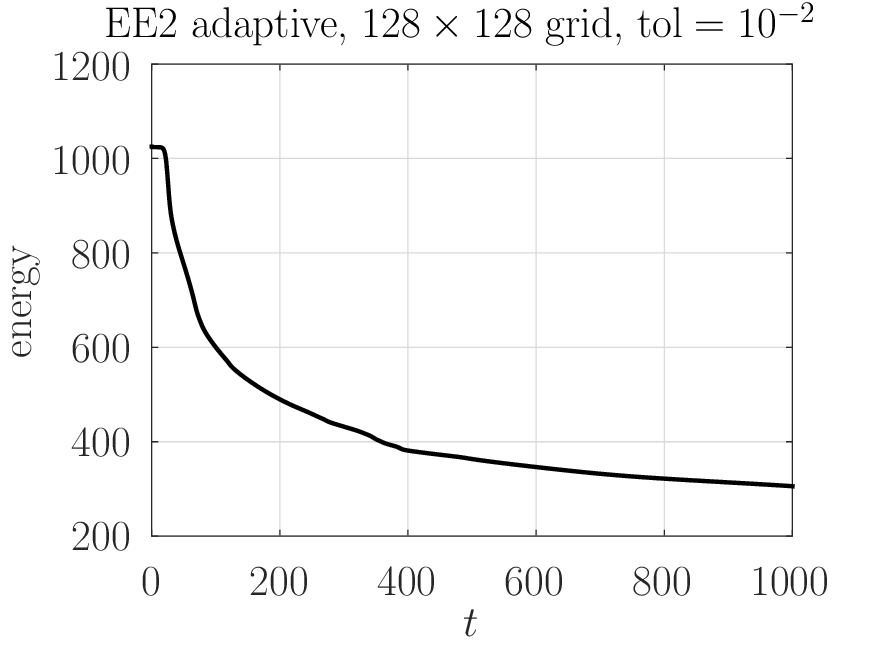}
\caption{Discrete energy versus time for the adaptive 
LIM($\tau_0=1$) (left) and EE2(10) (right) schemes run with $\tol=10^{-2}$}
\label{f:Ee2}
\end{figure}

For the same runs of the adaptive LIM($\tau_0=1$) and EE2(10) schemes
executed with the tolerance value~$10^{-2}$,
mass deviation plots are given in Figure~\ref{f:m2}.  As we see, both scheme
demonstrate quite acceptable mass conservation properties
even with this relaxed tolerance value.

\begin{figure}
\includegraphics[width=0.49\linewidth]{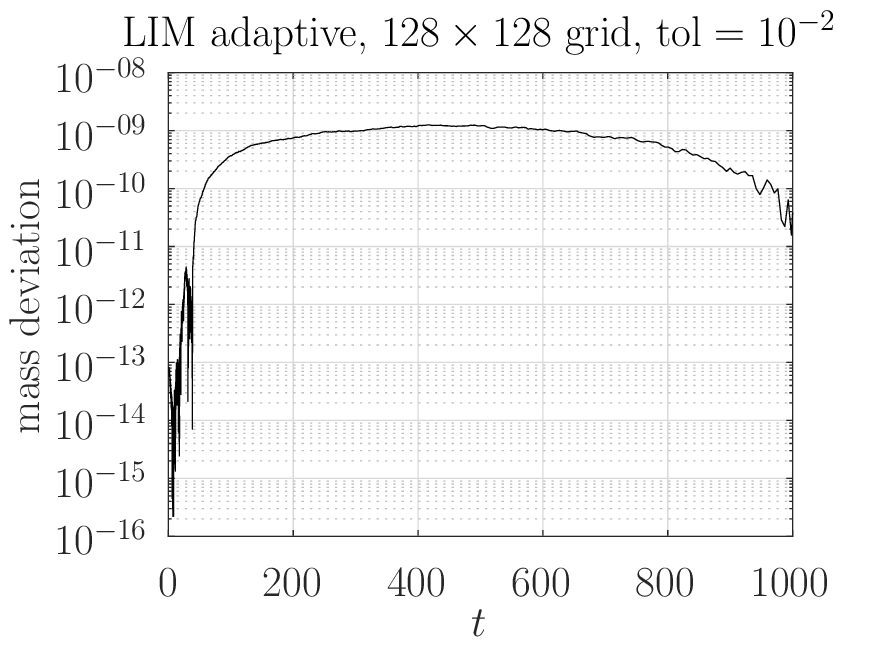}
\includegraphics[width=0.49\linewidth]{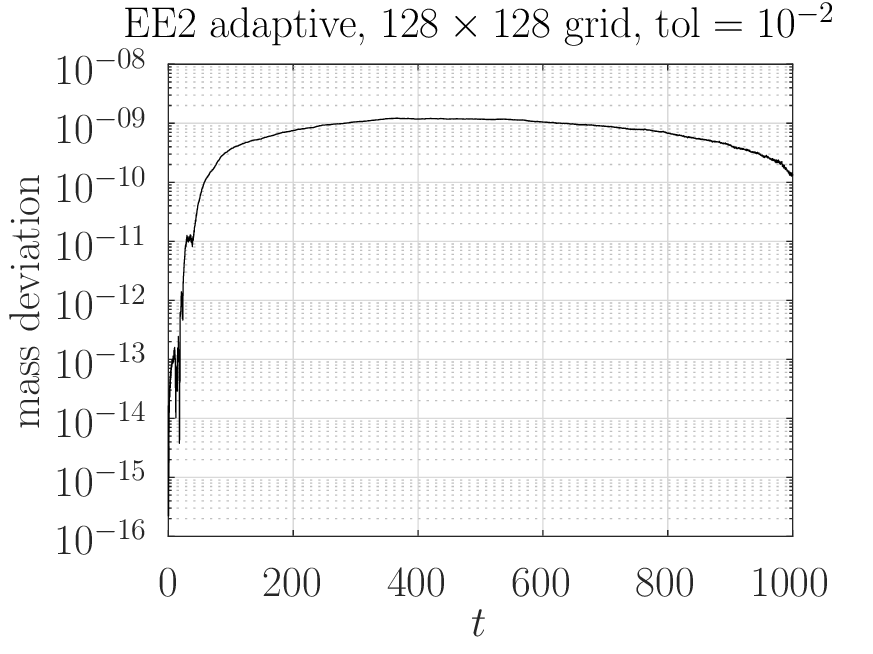}
\caption{Mass deviation $|(m^n-m^0)/m^0-1|$, with $m^n=\sum_jy_j^n$, 
versus time for the adaptive 
LIM($\tau_0=1$) (left) and EE2(10) (right) schemes run with $\tol=10^{-2}$}
\label{f:m2}
\end{figure}

Comparing the data in Tables~\ref{t:dt_const} and~\ref{t:var_dt128}
for the $128\times 128$ space grid, we see that the adaptive time step
size selection yields an efficiency gain for the LIM scheme by approximately
a factor~3.5 (for instance, 38940 instead of 136000 matrix-vector products to
reach roughly the same accuracy).  For the EE2 scheme the gain factor
is up to approximately a factor~2 (e.g., 73867 instead of 
141148 matrix-vector products to reach an accuracy $\approx 10^{-4}$).  

\section{Conclusions and an outlook to further research}
\label{s:concl}
In this paper an exponential time stepping scheme 
is proposed for numerical time integration of the Cahn--Hilliard equation.
The scheme has a nonstiff accuracy order~2 and we call it 
exponential Euler order~2 scheme (EE2).   
It is explicit in the sense that it does not require linear 
system solution.  We compare it to another explicit scheme, the 
local iteration modified (LIM) scheme~\cite{LokLokDAN,ShvedovZhukov,Zhukov2011}, 
which is recently introduced for the Cahn--Hilliard equation 
in~\cite{BotchevFahurdinovSavenkov2024}.  For both schemes we show how
an adaptive time step size selection, similar to recently proposed
in~\cite{BotchevZhukov2024}, can be implemented.  We also describe
how to combine both EE2 and LIM schemes with a variant of the
Eyre convex splitting, for which we show similarity of the
involved matrix to a symmetric positive semidefinite matrix.  

Results of numerical experiments presented here allow to make the following
conclusions.
\begin{enumerate}
\item 
When employed with a constant time step size and without the Eyre convex
splitting, the LIM scheme works robustly, whereas the EE2 scheme has 
problems with growing Krylov subspace dimension due to instability
(negative eigenvalues appearing in the Jacobian).  
With the maximal Krylov subspace dimension set
to an unrealistically large value $\mmax=100$, the EE2 scheme 
delivers a significantly higher time integration
accuracy than the LIM scheme for a comparable number of matrix-vector
multiplications, see Table~\ref{t:dt_const}.

\item
The Eyre convex splitting removes instability problems and allows
to set the maximal Krylov subspace dimension in EE2 to a realistic 
value (we had $\mmax=30$ in the tests).
However, combination with the Eye convex splitting significantly decreases 
accuracy in both the LIM and EE2 schemes.
The EE2 scheme combined with the Eyre convex splitting appears to
be more efficient than the LIM scheme combined with the Eyre splitting,
see Figure~\ref{f:Eyre}.

\item
The proposed adaptive time step size selection leads to an essential 
efficiency gain in both schemes: a gain factor is 
approximately 3.5 for the LIM scheme and around 2 for the EE2 scheme
(compare Tables~\ref{t:dt_const} and~\ref{t:var_dt128}).
The adaptive time step size selection also allows the EE2 scheme to deal 
with instability caused by small negative eigenvalues, while keeping 
the Krylov subspace dimension moderate ($\mmax=10$ and $\mmax=30$ in the tests)
and without the Eyre splitting stabilization.

\item
For low accuracy requirements ($\sim 10^{-3}$) 
the adaptive LIM scheme is slightly more efficient than the adaptive 
EE2 scheme with $\mmax=10$, see Figure~\ref{f:acc_vs_mv}.
For higher accuracy requirements ($\leqs 10^{-4}$) 
the adaptive EE2($\mmax=10$) scheme 
outperforms the adaptive LIM scheme. 
The adaptive EE2($\mmax=30$) scheme is essentially more efficient
than the adaptive LIM scheme.  

\item
The LIM scheme does not require any parameter tuning, whereas 
the tolerance in the Krylov subspace method of the EE2 scheme 
has to be selected carefully, see formula~\eqref{tol_phi}.  
In overall, comparing the LIM and EE2 schemes, we conclude 
the EE2 scheme allows to increase efficiency significantly (up to 
a factor~4 less matrix-vector multiplications), 
at the cost of storing and handling $\mmax$ Krylov subspace vectors.
\end{enumerate}

Future research can be aimed at improving the EE2 scheme, in particular,
with respect to robustness of its Krylov subspace method
for matrices having small negative ``unstable'' eigenvalues.
A question whether higher order methods can be applicable and
efficient for large-scale Cahn--Hilliard problems seems to be another
promising research direction.  

\section*{Acknowledgments}
The author thanks his colleagues
Vladislav Balashov, Ilyas Fahurdinov, Leonid Knizhnerman
and Evgenii Savenkov for fruitful discussions.

\section{Funding}
The authors of this work declare that their work was not financially
supported.

\section{CONFLICT OF INTEREST}
The authors of this work declare that they have no conflicts of
interest.

\def\ocirc#1{\ifmmode\setbox0=\hbox{$#1$}\dimen0=\ht0 \advance\dimen0
  by1pt\rlap{\hbox to\wd0{\hss\raise\dimen0
  \hbox{\hskip.2em$\scriptscriptstyle\circ$}\hss}}#1\else {\accent"17 #1}\fi}
  \def\cprime{$'$}


\end{document}